\documentclass[12pt,twoside,english]{iopart}
\usepackage[T1]{fontenc}
\usepackage[latin9]{inputenc}
\usepackage{xcolor}
\usepackage{pdfcolmk}
\usepackage{babel}
\usepackage{mathrsfs}
\usepackage{amsthm}
\usepackage{amstext}
\usepackage{graphicx}
\usepackage{esint}
\PassOptionsToPackage{normalem}{ulem}
\usepackage{ulem}
\usepackage[unicode=true,
 bookmarks=true,bookmarksnumbered=false,bookmarksopen=false,
 breaklinks=false,pdfborder={0 0 1},backref=false,colorlinks=false]
 {hyperref}

\makeatletter

\providecolor{lyxadded}{rgb}{0,0,1}
\providecolor{lyxdeleted}{rgb}{1,0,0}

\usepackage{iopams}
\usepackage{setstack}
\theoremstyle{plain}
\newtheorem{thm}{\protect\theoremname}[section]
  \theoremstyle{plain}
  \newtheorem{lem}[thm]{\protect\lemmaname}
  \theoremstyle{plain}
  \newtheorem{cor}[thm]{\protect\corollaryname}
  \theoremstyle{remark}
  \newtheorem{claim}[thm]{\protect\claimname}

\usepackage{mathrsfs}

 \theoremstyle{remark}
 \newtheorem{condition}[thm]{Condition}
\usepackage{cite}

\usepackage{bbm}

\@ifundefined{showcaptionsetup}{}{%
 \PassOptionsToPackage{caption=false}{subfig}}
\usepackage{subfig}
\makeatother

  \providecommand{\claimname}{Claim}
  \providecommand{\corollaryname}{Corollary}
  \providecommand{\lemmaname}{Lemma}
\providecommand{\theoremname}{Theorem}

\begin{document}

\title[Identifying Conditional Probability Measures]{Inverse Problems for a Class of Conditional Probability Measure-Dependent
Evolution Equations}

\author{David M.~Bortz$^{1}$, Erin C.~Byrne$^{2}$, and Inom Mirzaev$^{1}$}

\address{$^{1}$ Department of Applied Mathematics, University of Colorado,
Boulder, CO 80309-0526\\
$^{2}$ The MathWorks, Inc., 3 Apple Hill Drive, Natick, MA 01760}

\ead{\mailto{dmbortz@colorado.edu}, \mailto{erin.byrne@mathworks.com}, \mailto{mirzaev@colorado.edu}}
\begin{abstract}
We investigate the inverse problem of identifying a conditional probability
measure in a measure-dependent dynamical system. We provide existence
and well-posedness results and outline a discretization scheme for
approximating a measure. For this scheme, we prove general method
stability. 

The work is motivated by Partial Differential Equation (PDE) models
of flocculation for which the shape of the post-fragmentation conditional
probability measure greatly impacts the solution dynamics. To illustrate
our methodology, we apply the theory to a particular PDE model that
arises in the study of population dynamics for flocculating bacterial
aggregates in suspension, and provide numerical evidence for the utility
of the approach.
\end{abstract}



\submitto{~}

\maketitle

\section{Introduction}

In this paper, we examine an inverse problem involving a general measure-dependent
partial differential equation (PDE). We consider a general abstract
evolution equation with solution $b$, depending on the conditional
probability measure $F$:
\begin{eqnarray}
\qquad b_{t} & = & g(b,\,F)\label{eq:AbstractEvEqn}\\
b(0,x) & = & b_{0}(x)\label{eq:AbstractEvEqnIC}
\end{eqnarray}

\noindent for $t\in T=[0,\,t_{f}]$ with $t_{f}<\infty$. As investigated
in our previous work \cite{Bortz2008,Mirzaev2015b}, the function
space for both the initial condition $b_{0}(\cdot)$ and the solution
$b(t,\,x)$ is $H$$=L^{1}(Q,\,\mathbb{R}^{+})$, where $Q=[0,\,\overline{x}]$,
$\overline{x}\in\mathbb{R}^{+}$ and $g:H\times\mathscr{F}\to H$.

This study of this class of models is motivated by our interest in
studying fragmentation phenomena, which arise in a wide variety of
areas including size structured algal populations \cite{AcklehFitzpatrick1997,Ackleh1997,Lamb2009},
cancer metastases \cite{DeVitaTheodore2008,IlanaElkinbVlodavsk2006,Wyckoff2000},
and mining \cite{Gamma1993,Persson1994}. In \cite{Bortz2008}, we
developed a size-structured partial differential equation (PDE) model
for bacterial \emph{flocculation}, the process whereby aggregates,
i.e., \emph{flocs}, in suspension adhere and separate. For the breakage
term in that PDE model each fragmentation event will generate child
particles according to a \emph{post-fragmentation probability distribution}.
In the literature, it is widespread to assume that this distribution
is independent of parent floc size and is normally distributed. However,
in \cite{ByrneEtal2011}, we focused only on the fragmentation and
developed a microscale mathematical model which contradicts this result
and predicted that the distribution is both dependent on parent size
and non-normal. Thus it is clear that there is a need for a methodology
to identify this conditional distribution from available data.

In this work, we present and investigate an inverse problem for estimating
the conditional probability measures from size-distribution measurements.
We use the Prohorov metric (convergence in which is equivalent to
weak convergence of measures) in a functional-analytic setting and
show well-posedness of the inverse problem. We develop an approximation
approach for computational implementation and show well-posedness
of this approximate inverse problem. We also show the convergence
of solutions to the approximate inverse problem to solutions of the
original inverse problem. Our approach is inspired by that for identifying
a single probability measure in Banks and Bihari \cite{BanksBihari2001}
and a countable number of probability measures in Banks and Bortz
\cite{BanksBortz2005JIIP}. The primary contribution of this work
is to extend this theory to conditional probability measures. We also
illustrate that the flocculation dynamics of bacterial aggregates
in suspension is one realization of systems satisfying the hypotheses
in our framework.

\section{Well-Posedness of the Inverse Problem}

We begin by considering the model in (\ref{eq:AbstractEvEqn})-(\ref{eq:AbstractEvEqnIC}).
In this section, we will develop the theoretical results needed to
prove the well-posedness of the inverse problem.

Note that our eventual goal is to infer the post-fragmentation distribution
$F(x,\,y)$ from laboratory data. Accordingly, we will make some assumptions
which are driven by the features of the available validating data.

\subsection{Theoretical framework}

Let $\mathscr{P}(Q)$ be the space of all probability distributions
on $(Q,\,\mathscr{A})$, where $\mathscr{A}$ is the Borel $\sigma$-algebra
on $Q$. Since we are primarily concerned with the system in (\ref{eq:AbstractEvEqn})-(\ref{eq:AbstractEvEqnIC}),
we restrict the space of probability distributions to those that can
be solutions to our inverse problem. A fragmentation cannot result
in a daughter floc larger than the original floc, therefore we consider
the subset $\mathscr{P}_{y}(Q)\subset\mathscr{P}(Q)$ such that $F(x,\,y)\in\mathscr{P}_{y}(Q)$
if $F(x,\,y)\equiv1$ for $x\geq y$ and fixed $y\in Q$. We also
restrict our solutions to piecewise absolutely continuous (PAC) functions
with a finite number of discontinuities in $x$ for a fixed $y$.
An illustration of the domain and an example using a Beta distribution
($\alpha=\beta=2$) is depicted in Figures \ref{fig:Domain} and \ref{fig:Domain-1}.
Note that in Figure \ref{fig:Domain}, the upper left corner of the
domain admits values of $F(x,\,y)$ between 0 and 1, and the lower
right requires $F(x,\,y)\equiv1$. We then define our space of solutions
to the inverse problem as $\mathscr{F}(Q\times Q)$, the space of
all PAC functions with a finite number of discontinuities in $x$
such that $F(x,\,y)\in\mathscr{P}_{y}(Q)$ for any fixed $y$.
\begin{figure}[ht]
\centering{}\subfloat[\label{fig:Domain}]{\protect\includegraphics[scale=0.4]{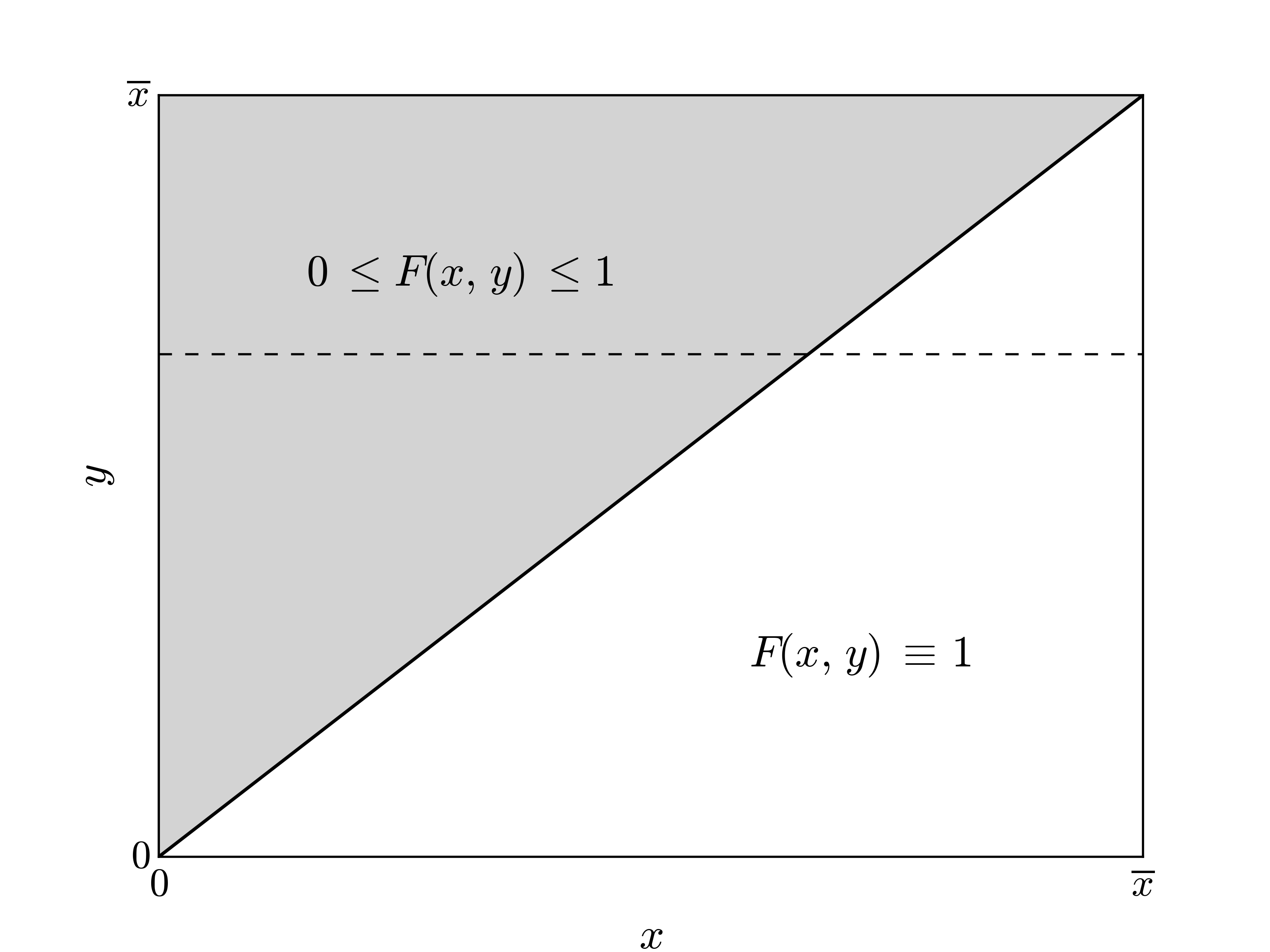}}~~\subfloat[\label{fig:Domain-1}]{\protect\includegraphics[scale=0.4]{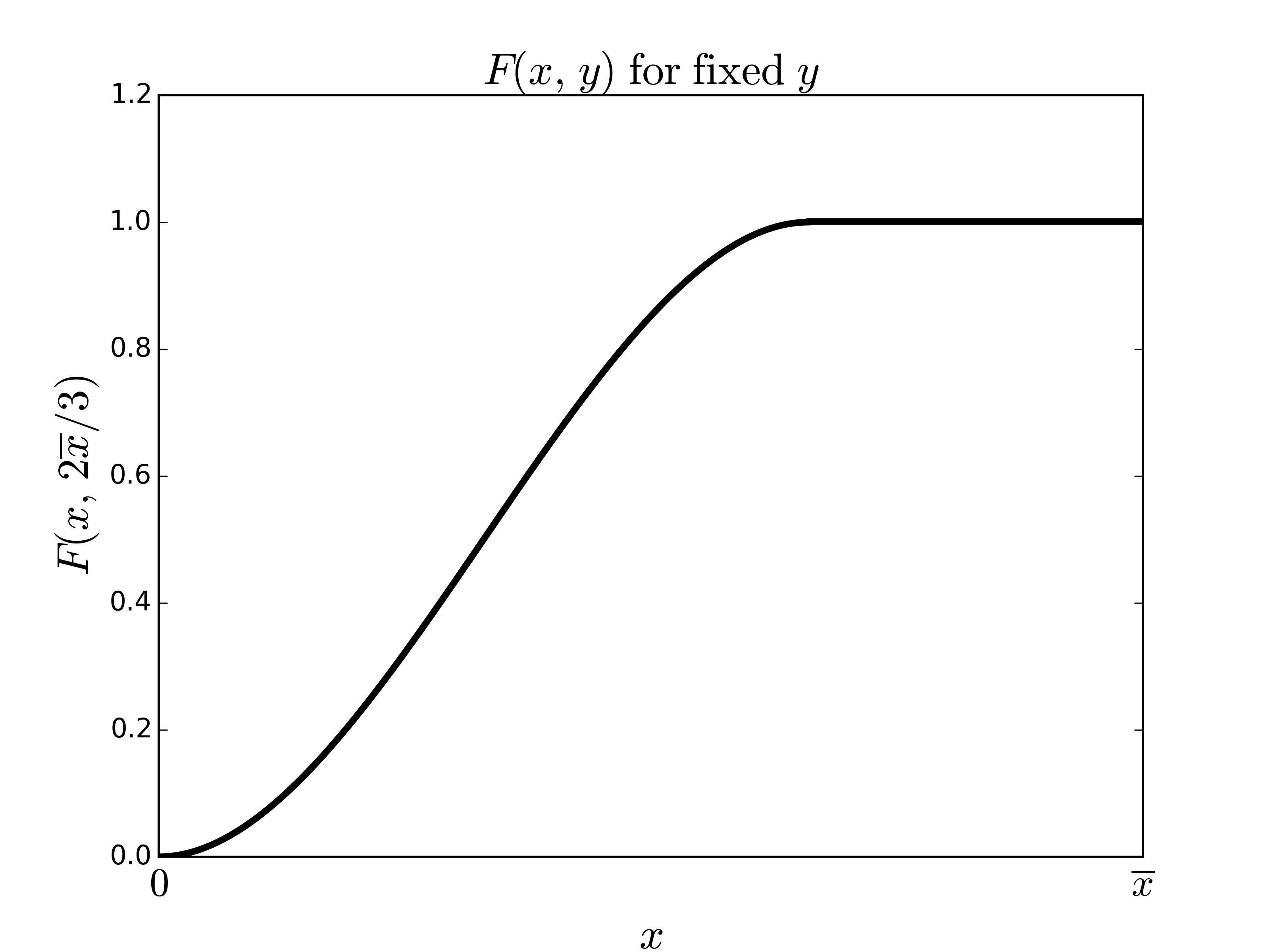}}\protect\caption{(a) Domain for the probability measure $F(x,\,y)$ showing admissible
values for $F(x,\,y)\subset\mathscr{P}(Q)$. Dotted line is for $y=2\overline{x}/3$.
(b) Example $F(x,\,2\overline{x}/3)\subset\mathscr{P}(Q)$ for fixed
$y$. In this case, $\Gamma$ is a Beta distribution with $\alpha=\beta=2$ }
\end{figure}

We define a metric on the space $\mathscr{F}$ to create a metric
topology, and we accomplish this by making use of the well-known Prohorov
metric (see \cite{Billingsley1968} for a full description). Convergence
in the Prohorov metric is equivalent to weak convergence, and we direct
the interested reader to \cite{Gibbs2009} for a summary of its relationship
to a variety of other metrics on probability measures. For $F,\,\tilde{F}\in\mathscr{F}$
and fixed $y$, we use the Prohorov metric $\rho_{Proh}$ to denote
the distance $\rho_{Proh}(F(\cdot,y),\,\tilde{F}(\cdot,y))$ between
the measures. We extend this concept to define the metric $\rho$
on the space $\mathscr{F}(Q\times Q)$ by taking the supremum of $\rho_{Proh}$
over all $y\in Q$, 
\[
\rho(F,\,\tilde{F})=\sup_{y\in Q}\,\rho_{Proh}(F(\cdot,\,y),\,\tilde{F}(\cdot,\,y))\,.
\]

The most widely available, high-fidelity data for flocculating particles
are in the form of particle size histograms from, e.g., from flow-cytometers,
Coulter counters, etc. Accordingly, we will define our inverse problem
with the goal of comparing with histograms of floc sizes. Let $n_{j}(t_{i})$
represent the number of flocculated biomasses with volume between
$x_{j}$ and $x_{j+1}$ at time $t_{i}$. We assume that the data
is generated by an actual post-fragmentation function. In other words,
$\mathbf{n}^{d}$ is representable as the partial zeroth moment of
the solution 
\[
n_{ji}^{d}=\int_{x_{j-1}}^{x_{j}}b(t_{i},\,x;F_{0})\,dx+\mathcal{E}_{ji}
\]
for some \emph{true} probability-measure $F_{0}\in\mathscr{F}$. The
random variables $\mathcal{E}_{ji}$ represent measurement noise.
We also assume, as it is commonly assumed in statistics, that the
random variables $\mathcal{E}_{ji}$ are independent, identically
distributed, $E[\mathcal{E}_{ji}]=0$ and $Var[\mathcal{E}_{ji}]=\sigma^{2}<\infty$
(which is generally true for flow-cytometers \cite{Darzynkiewicz}).
Thus our inverse problem entails finding a minimizer of the least
squares cost functional, defined as

\begin{equation}
\min_{F\in\mathscr{F}}\;J(F;\,\mathbf{n}^{d})=\min_{F\in\mathscr{F}}\;\sum_{i=1}^{N_{t}}\sum_{j=1}^{N_{x}}\left(\int_{x_{j-1}}^{x_{j}}b(t_{i},x;F)\,dx\,-\,n_{ji}^{d}\right)^{2}\,,\label{eq:OrigInv}
\end{equation}

\noindent where the data $\mathbf{n}^{d}\in\mathbb{R}^{N_{x}\times N_{t}}$
consists of the number of flocs in each of the $N_{x}$ bins for floc
volume at $N_{t}$ time points. The superscript $d$ denotes the dimension
of the data, $d=N_{x}\times N_{t}$. The function $b$ is the solution
to (\ref{eq:AbstractEvEqn})-(\ref{eq:AbstractEvEqnIC}) corresponding
to the probability measure $F$. 

For a given data $\mathbf{n}^{d}$, the cost function $J$ may not
have a unique minimizer, thus we denote a corresponding solution set
of probability distributions as $\mathscr{F}^{*}(\mathbf{n}^{d}).$
We then define the distance between two such sets of solutions, $\mathscr{F}^{*}(\mathbf{n}_{1}^{d_{1}})$
and $\mathscr{F}^{*}(\mathbf{n}_{2}^{d_{2}})$ (for data $\mathbf{n}_{1}^{d_{1}}$
and $\mathbf{n}_{2}^{d_{2}}$) to be the well-known Hausdorff distance
\cite{Kelley1955}
\[
d_{H}(\mathscr{F}^{*}(\mathbf{n}_{1}^{d_{1}}),\,\mathscr{F}^{*}(\mathbf{n}_{2}^{d_{2}}))=\inf\{\rho(F,\tilde{\,F}):\;F\in\mathscr{F}^{*}(\mathbf{n}_{1}^{d_{1}}),\;\tilde{F}\in\mathscr{F}^{*}(\mathbf{n}_{2}^{d_{2}})\}\,.
\]

\subsection{Inverse Problem}

\noindent In this section we will establish well-posedness of the
inverse problem defined in (\ref{eq:OrigInv}). In particular, we
will first show that for a given data $\mathbf{n}^{d}$ with dimension
$d$ the least squares estimator defined in (\ref{eq:OrigInv}) has
at least one minimizer. Next, we will investigate the behavior of
minimizers of (\ref{eq:OrigInv}) as more data is collected. Specifically,
we will show that the least squares estimator is consistent, i.e.,
as the dimension of data increases ($N_{t}\to\infty$ and $N_{x}\to\infty$)
the minimizers of the estimator (\ref{eq:OrigInv}) converge to \emph{true}
probability measure $F_{0}$ generating the data $\mathbf{n}^{d}$.

\subsubsection{Existence of the estimator.}

\noindent In this section we prove that the cost functional defined
in (\ref{eq:OrigInv}) possesses at least one minimizer. We use the
well-known result that a continuous function on a compact metric space
has a minimum. In particular, first we show that $(\mathscr{F},\,\rho)$
is a compact metric space. Next, we establish continuous dependence
of the solution $b$ on the conditional probability measure $F$. 

For much of the following analysis, we require the operator $g$ to
satisfy a Lipschitz-type condition. We detail that condition in the
following.

\begin{condition}\label{con:gLipschitz} Suppose that $b$ and $\tilde{b}$
are solutions to the evolution equation (\ref{eq:AbstractEvEqn})-(\ref{eq:AbstractEvEqnIC}).
For fixed $t$, the function $g:\,H\times\mathscr{F}\to H$ must satisfy
\[
\left\Vert g(b,\,F)-g(\tilde{b},\,\tilde{F})\right\Vert \leq C\left\Vert b-\tilde{b}\right\Vert +\mathscr{T}(F,\tilde{\,F}),
\]
 where $C>0$, and $\mathscr{T}(F,\,\tilde{F})$ is some function
such that $|\mathscr{T}(F,\,\tilde{F})|<\infty$ and $\mathscr{T}(F,\,\tilde{F})\to0$
as $\rho(F,\,\tilde{F})\to0$.\end{condition}

We begin by proving that $(\mathscr{F},\,\rho)$ is a compact metric
space.
\begin{lem}
$(\mathscr{F},\,\rho)$ is a compact metric space. \end{lem}
\begin{proof}
Consider a Cauchy sequence $\{F_{n}\}\in\mathscr{F}$. Then $\forall\;\epsilon>0,\;\exists\,N$
such that $\forall\,n,m\geq N$, 
\[
\sup_{y\in Q}\;\rho_{Proh}\left(F_{n}(\cdot,\,y),\,F_{m}(\cdot,\,y)\right)\,<\,\epsilon.
\]
It is easy to see we have a Cauchy sequence $\{F_{n}(\cdot,\,y)\}\in\mathscr{P}_{y}(Q)$
which converges uniformly in $y\in Q$. From results in Billingsley
\cite{Billingsley1968}, $\mathscr{P}_{y}(Q)$ is a compact metric
space, there exists $F(\cdot,\,y)\in\mathscr{P}_{y}(Q)$ such that
$\rho_{Proh}(F_{n}(\cdot,\,y),\,F(\cdot,\,y))<\epsilon$ for all $n\geq N$.
Thus 
\[
\sup_{y\in Q}\;\rho_{Proh}\left(F_{n}(\cdot,\,y),\,F(\cdot,\,y)\right)\,<\,\epsilon
\]
and $(\mathscr{F},\,\rho)$ is a complete metric space. In addition,
since $Q\times Q$ is compact and $0\leq F(x,\,y)\leq1$ for all $(x,\,y)\in Q\times Q,\,F\in\mathscr{F}$,
$\mathscr{F}(Q\times Q)$ is totally bounded and therefore $(\mathscr{F},\,\rho)$
is a compact metric space.
\end{proof}
\noindent Now that we have a compact metric space, it remains to show
that the cost functional on that space is continuous with respect
to the function $F$. It suffices to prove point-wise continuity.
\begin{lem}
\noindent \label{lem:bcontF}If $t\in T$, $F\in\mathscr{F}$, and
the operator $g$ in (\ref{eq:AbstractEvEqn}) satisfies Condition~\ref{con:gLipschitz},
then the unique solution $b$ to (\ref{eq:AbstractEvEqn}) is point-wise
continuous at $F\in\mathscr{F}.$ Moreover, since $\mathscr{F}$ is
compact space the unique solution $b$ is uniformly continuous on
$\mathscr{F}$.\end{lem}
\begin{proof}
For the function $b$ to be point-wise continuous at $F$, we need
to show that $\left\Vert b(t,\cdot;F_{i})-b(t,\cdot;F)\right\Vert \to0$
as $\rho(F_{i},F)\to0$ for $\{F_{i}\}\in\mathscr{F}$ and fixed $t$.
We begin by re-writing (\ref{eq:AbstractEvEqn}) as an integral equation
\[
b(t,x)=b_{0}(x)+\int_{0}^{t}g(b(s,x),F)ds\,.
\]
For fixed $t$, consider $b$ to be a function of $F$ 
\[
b(t,x;F)=b_{0}(x)+\int_{0}^{t}g(b(s,x;F),F)ds\,.
\]
By definition of solutions, we have 
\[
\left\Vert b(t,\cdot;F_{i})-b(t,\cdot;F)\right\Vert \leq\int_{0}^{t}\left\Vert g(b(s,\cdot;F_{i}),F_{i})-g(b(s,\cdot;F),F)\right\Vert ds\,.
\]
Based on Condition \ref{con:gLipschitz}, we obtain 
\[
\left\Vert b(t,\cdot;F_{i})-b(t,\cdot;F)\right\Vert \leq C\int_{0}^{t}\left\Vert b(s,\cdot;F_{i})-b(s,\cdot;F)\right\Vert ds+\mathcal{T}(F_{i},F)\,,
\]
where we define $\mathcal{T}(F_{i},F)=\int_{0}^{t_{f}}\mathscr{T}(F_{i},F)ds$,
independent of $t$. An application of Gronwall's inequality yields
\[
\left\Vert b(t,\cdot;F_{i})-b(t,\cdot;F)\right\Vert \leq\mathcal{T}(F_{i},F)\,e^{\int_{0}^{t}C\,ds}\leq\mathcal{T}(F_{i},F)\,e^{C\,t_{f}}\to0
\]
since we know that $\mathcal{T}(F_{i},F)\to0$ as $F_{i}\to F$ in
$(\mathscr{F},\,\rho)$. Thus the solutions $b$ are point-wise continuous
at $F\in\mathscr{F}$. 
\end{proof}
We use the results of the above two lemmas to establish existence
of a solution to our inverse problem.
\begin{thm}
\label{thm:OrigExist}There exists a solution to the inverse problem
as described in (\ref{eq:OrigInv}).\end{thm}
\begin{proof}
It is well known that a continuous function on a compact set obtains
both a maximum and a minimum. We have shown $(\mathscr{F},\rho)$
is compact, and from Lemma~\ref{lem:bcontF}, for fixed $t\in T,$
we have that $F\mapsto b(t,\cdot;F)$ is continuous. Since $J$ is
continuous with respect to $F$ and we can conclude there exist minimizers
for $J$.
\end{proof}

\subsubsection{Consistency of the estimator.}

In previous section we have proved that for a given data there exists
estimators for the least squares problem. In this section we will
investigate the behavior of the least squares estimators as the number
of observations increase. In particular, the estimator is said to
be \emph{consistent} if the estimators for the data $\mathbf{n}^{d}$
converge to \emph{true} probability measure $F_{0}$ as $N_{t}\to\infty$
and $N_{x}\to\infty$. Consistency of the estimators of the least
squares problems are well-studied in the statistics and the results
of this section follow closely the theoretical results of \cite{BanksThompson2012}
and \cite{htb1990}. Hence, as in \cite[Theorem 4.3]{BanksThompson2012}
and \cite[Corallary 3.2]{htb1990}, we will make the following two
assumptions required for the convergence of the estimators to the
unique \emph{true} probability measure $F_{0}$. \begin{itemize}

\item[(\textbf{A}1)]Let us denote the space of positive functions
$T\times Q\mapsto\mathbb{R}^{+}$, which are bounded and Riemann integrable
by $\mathscr{R}\left(T\times Q,\,\mathbb{R}^{+}\right)$. Then, the
model function $b(t,x;\,\cdot)\,:\,\mathscr{F}\to\mathscr{R}\left(T\times Q,\,\mathbb{R}^{+}\right)$
is continuous on $\mathscr{F}$. 

\item[(\textbf{A}2)]The functional 
\[
J_{0}(F)=\sigma^{2}+\int_{T}\int_{Q}\left(b(t,x;F)-b(t,x;F_{0})\right)^{2}\,dx\,dt
\]
is uniquely (up to $L^{1}$ norm) minimized at $F_{0}\in\mathscr{F}$.\end{itemize}

Having the required assumptions in hand, we now present the following
theorem.
\begin{thm}
\label{thm:Consistency of the estimator}Under assumptions (\textbf{A}1)
and (\textbf{A}2) 
\[
d_{H}\left(\mathscr{F}^{*}(\mathbf{n}^{d}),\,F_{0}\right)\to0
\]
as $N_{t}\to\infty$ and $N_{x}\to\infty$. \end{thm}
\begin{proof}
The specific details of this proof are nearly identical to a similar
theorem in \cite{BanksThompson2012} and so here we simply provide
an overview. Briefly, one first shows that $J(F;\,\mathbf{n}^{d})$
converges to $J_{0}(F)$ for each $F\in\mathscr{F}$ as $N_{t}\to\infty$
and $N_{x}\to\infty$. Then, using the fact that $J_{0}(F)$ is uniquely
minimized at $F_{0}$, one can show that for each sequence $\left\{ F^{d}\in\mathscr{F}^{*}(\mathbf{n}^{d})\right\} $
the Prohorov distance $\rho\left(F^{d},\,F_{0}\right)$ converges
to zero as $N_{t}\to\infty$ and $N_{x}\to\infty$, which yields the
result.
\end{proof}

\section{Approximate Inverse Problem}

Since the original problem involves minimizing over the infinite dimensional
space $\mathscr{F}$, pursuing this optimization is challenging without
some type of finite dimensional approximation. Thus we define some
approximation spaces over which the optimization problem becomes computationally
tractable. Similar to the partitioning presented in \cite{BanksBortz2005JIIP},
let $Q_{M}=\{q_{j}^{M}\}_{j=0}^{M}$ be partitions of $Q=[0,\overline{x}]$
for $M=1,2,\dots$ and

\begin{equation}
Q_{D}=\bigcup_{M=1}^{\infty}Q_{M}\label{eq:QD}
\end{equation}

\noindent where the sequences are chosen such that $Q_{D}$ is dense
in $Q$.

For positive integers $M,\,L$, let the approximation space be defined
as 
\begin{eqnarray*}
\mathscr{F}^{ML} & = & \left\{ F\in\mathscr{F}\;|\;F(x,y)=\sum_{m=1}^{M}p_{\ell m}\Delta_{q_{m}^{M}}(x)\,\mathbbm{1}_{(q_{\ell-1}^{L},\,q_{\ell}^{L}]}(y),\right.\\
 &  & \left.\qquad q_{m}^{M}\in Q_{M},\,q_{\ell}^{L}\in Q_{L},\,\sum_{m=1}^{\ell}p_{\ell m}=1,\,\ell=1,2,\dots,L\right\} 
\end{eqnarray*}

\noindent where $\Delta_{q}(x)$ is the Heaviside step function with
atom $x=q$ and the function $\mathbbm{1}_{A}$ is the indicator function
on the interval $A$. Next, define the space $\mathscr{F}_{D}$ as
\[
\mathscr{F}_{D}=\bigcup_{M,L=1}^{\infty}\mathscr{F}^{ML}.
\]
Consequently, since $Q$ is a complete, separable metric space, and
by Theorem 3.1 in \cite{BanksBihari2001} and properties of the sup
norm, $\mathscr{F}_{D}$ is dense in $\mathscr{F}$ in the $\rho$
metric. Therefore we can directly conclude that any function $F\in\mathscr{F}$
can be approximated by a sequence $\{F_{M_{j}L_{k}}\}$, $F_{M_{j}L_{k}}\in\mathscr{F}^{M_{j}L_{k}}$
such that as $M_{j},\,L_{k}\to\infty$, $\rho(F_{M_{j}L_{k}},\,F)\to0$.

Similar to the discussion concerning Theorem 4.1 in \cite{BanksBihari2001},
we now state the theorem regarding the continuous dependence of the
inverse problem upon the given data, as well as stability under approximation
of the inverse problem solution space $\mathscr{F}$.
\begin{thm}
\label{thm:Orig Cont dep on data} Let $Q=[0,\overline{x}]$, assume
that for fixed $t\in T,\;x\in Q$, $F\mapsto b(t,x,F)$ is continuous
on $\mathscr{F}$, and let $Q_{D}$ be a countable dense subset of
$Q$ as defined in (\ref{eq:QD}). Suppose that $\mathscr{F}^{*ML}(\mathbf{n}^{d})$
is the set of minimizers for $J(F;\mathbf{n}^{d})$ over $F\in\mathscr{F}^{ML}$
corresponding to the data $\mathbf{n}^{d}$. Then, $d_{H}(\mathscr{F}^{*ML}(\mathbf{n}^{d}),\,F_{0})\to0$
as $M,\,L,\,N_{t},\,N_{x}\to\infty$.\end{thm}
\begin{proof}
Suppose that $\mathscr{F}^{*}(\mathbf{n}^{d})$ is the set of minimizers
for $J(F;\mathbf{n}^{d})$ over $F\in\mathscr{F}$ corresponding to
the data $\mathbf{n}^{d}$. Using continuous dependence of solutions
on $F$, compactness of $(\mathscr{F},\,\rho)$, and the density of
$\mathscr{F}_{D}$ in $\mathscr{F}$, the arguments follow precisely
those for Theorem 4.1 in \cite{BanksBihari2001}. In particular, one
would argue in the present context that any sequence $F_{d}^{*ML}\in\mathscr{F}^{*ML}(\mathbf{n}^{d})$
has a subsequence $F_{d_{k}}^{*M_{j}L_{i}}$ that converges to a $\tilde{F}\in\mathscr{F}^{*}(\mathbf{n}^{d})$.
Therefore, we can claim that 
\begin{equation}
d_{H}\left(\mathscr{F}^{*ML}(\mathbf{n}^{d}),\,\mathscr{F}^{*}(\mathbf{n}^{d})\right)\to0\label{eq: approximate space converges}
\end{equation}
as $M,\,L,\,N_{t},\,N_{x}\to\infty$. Conversely, simple triangle
inequality yields that 
\[
d_{H}(\mathscr{F}^{*ML}(\mathbf{n}^{d}),\,F_{0})\le d_{H}\left(\mathscr{F}^{*ML}(\mathbf{n}^{d}),\,\mathscr{F}^{*}(\mathbf{n}^{d})\right)+d_{H}\left(\mathscr{F}^{*}(\mathbf{n}^{d}),\,F_{0}\right)\,.
\]
This is in turn, from (\ref{eq: approximate space converges}) and
Theorem \ref{thm:Consistency of the estimator}, implies that $d_{H}(\mathscr{F}^{*ML}(\mathbf{n}^{d}),\,F_{0})$
converges to zero as $M,\,L,\,N_{t},\,N_{x}\to\infty$.
\end{proof}
Since we do not have direct access to an analytical solution to (\ref{eq:AbstractEvEqn}),
our efforts are focused on the solving the approximate inverse problem
\begin{equation}
\min_{F\in\mathcal{F}}J^{N}(F,\mathbf{n}^{d})=\min_{f\in\mathcal{F}}\sum_{i=1}^{N_{t}}\sum_{j=1}^{N_{x}}\left(\int_{x_{j-1}}^{x_{j}}b^{N}(t_{i},x_{j};F)dx-n_{ji}^{d}\right)^{2}\,.\label{eq:ApproxInv}
\end{equation}
Here, $N_{t}$ is the number of data observations, $N_{x}$ is the
number of data bins for floc volume, and $b^{N}$ is the semi-discrete
approximation to $b$. In Section \ref{sec:Example-Illustration},
we will define a uniformly (in time) convergent discretization scheme
and its corresponding approximation space $H^{N}\subset H$. The discretized
version of (\ref{eq:ApproxInv}) is represented by 
\begin{eqnarray}
\qquad b_{t}^{N} & = & g^{N}(b^{N},F)\label{eq:AbstractEvEqn-1}\\
b^{N}(0,x) & = & b_{0}^{N}(x)\label{eq:AbstractEvEqnIC-1}
\end{eqnarray}
where $g^{N}:H^{N}\times\mathscr{F}\to H^{N}$ denotes the discretized
version of $g$. We will need that $g^{N}$ exhibits a type of local
Lipschitz continuity and accordingly define the following condition.

\begin{condition}\label{cond:gNLipschitz} Suppose that the discretization
given in (\ref{eq:AbstractEvEqn-1})-(\ref{eq:AbstractEvEqnIC-1})
is a convergent scheme. Let $(b^{N},F),\,(\tilde{b}^{N},\tilde{F})\in H^{N}\times\mathscr{F}$.
For fixed $t$, the function $g^{N}:\,H^{N}\times\mathscr{F}\to H^{N}$
must satisfy 
\[
\left\Vert g^{N}(b^{N},F)-g^{N}(\tilde{b}^{N},\tilde{F})\right\Vert \leq C_{N}\left\Vert b^{N}-\tilde{b}^{N}\right\Vert +\mathscr{T}^{N}(F,\tilde{F}),
\]
where $C_{N}>0$, and $\mathscr{T}^{N}(F,\tilde{F})$ is some function
such that $|\mathscr{T}^{N}(F,\tilde{F})|<\infty$ and $\mathscr{T}^{N}(F,\tilde{F})\to0$
as $\rho(F,\tilde{F})\to0$.\end{condition}

General \emph{method stability} \cite{BanksKunisch1989} requires
$b^{N}(t,x;F_{i})\to b(t,x;F)$ as $F_{i}\to F$ in the $\rho$ metric
and as $N\to\infty$; we will now prove this.
\begin{lem}
\label{lem:bNFi} Let $t\in T$, $F\in\mathscr{F},$and $\{F_{i}\}\in\mathscr{F}$
such that $\lim_{i\to\infty}\rho(F_{i},F)=0$. For fixed $N$, if
$b^{N}(t,x;F_{i})$ is the solution to (\ref{eq:9N})-(\ref{eq:9Nend})
and Condition \ref{cond:gNLipschitz} holds, then $b^{N}$ is pointwise
continuous at $F\in\mathscr{F}$.\end{lem}
\begin{proof}
The proof of this lemma is identical to that for Lemma \ref{lem:bcontF}.
We first recast (\ref{eq:AbstractEvEqn-1}) as an integral equation
and then apply Condition \ref{cond:gNLipschitz} and Gronwall's inequality
to obtain the desired result.\end{proof}
\begin{cor}
\label{cor:bNFNtobF} Under Condition \ref{cond:gNLipschitz} and
Lemma~\ref{lem:bNFi}, we can conclude that $\left\Vert b^{N}(t,\cdot;F_{N})-b(t,\cdot;F)\right\Vert \to0$
as $N\to\infty$ uniformly in $t$ on $I$.\end{cor}
\begin{proof}
A standard application of the triangle inequality yields 
\begin{eqnarray*}
\left\Vert b^{N}(t,\cdot;F_{N})-b(t,\cdot;F)\right\Vert  & \leq & \left\Vert b^{N}(t,\cdot;F_{N})-b^{N}(t,\cdot;F)\right\Vert \\
 &  & \quad+\left\Vert b^{N}(t,\cdot;F)-b(t,\cdot;F)\right\Vert \,.
\end{eqnarray*}
The first term converges by Lemma~\ref{lem:bNFi}, while the second
term converges because the proposed numerical scheme is assumed to
converge uniformly.
\end{proof}
With this corollary, we now consider the existence of a solution to
the approximate inverse problem in (\ref{eq:ApproxInv}), as well
as the solution's dependence on the given data $\mathbf{n}^{d}$.
\begin{thm}
\label{thm:Exist/conv of orig =000026 approx mins}There exists solutions
to both the original and approximate inverse problems in (\ref{eq:OrigInv})
and (\ref{eq:ApproxInv}), respectively. Moreover, for fixed data
$\mathbf{n}^{d}$, there exist a subsequence of the estimators $\left\{ F_{N}\right\} _{N=1}^{\infty}$
of (\ref{eq:ApproxInv}) that converge to a solution of the original
inverse problem (\ref{eq:OrigInv}).\end{thm}
\begin{proof}
As noted above, $(\mathscr{F},\rho)$ is compact. By Lemmas~\ref{lem:bcontF}
and \ref{lem:bNFi}, we have that both $F\mapsto b(t,x;F)$ and $F\mapsto b^{N}(t,x;F)$,
for fixed $t\in T$, are continuous with respect to $F$. We therefore
know there exist minimizers in $\mathscr{F}$ to the original and
approximate cost functionals $J$ and $J^{N}$ respectively.

Let $\{F_{N}^{*}\}\in\mathscr{F}$ be any sequence of solutions to
(\ref{eq:ApproxInv}) and $\{F_{N_{k}}^{*}\}$ a convergent (in $\rho$)
subsequence of minimizers. Recall that minimizers are not necessarily
unique, but one can always select a convergent subsequence of minimizers
in $\mathscr{F}$. Denote the limit of this subsequence with $F^{*}$.
By the minimizing properties of $F_{N_{k}}^{*}\in\mathscr{F}$, we
then know that

\begin{equation}
J^{N_{k}}(F_{N_{k}}^{*},\mathbf{n}^{d})\leq J^{N_{k}}(F,\mathbf{n}^{d})\qquad\text{for all }F\in\mathscr{F}.\label{eq:JNkFNk leq JNkF}
\end{equation}

By Corollary~\ref{cor:bNFNtobF}, we have the convergence of $b^{N}(t,x;F_{N})\to b(t,x;F)$
and thus $J^{N}(F_{N})\to J(F)$ as $N\to\infty$ when $\rho(F_{N},F)\to0$.
Thus in the limit as $N_{k}\to\infty$, the inequality in (\ref{eq:JNkFNk leq JNkF})
becomes 
\[
J(F^{*},\mathbf{n}^{d})\leq J(F,\mathbf{n}^{d})\qquad\text{for all }F\in\mathscr{F}
\]
 with $F^{*}$ providing a (not necessarily unique) minimizer of (\ref{eq:OrigInv}).\end{proof}
\begin{thm}
Assume that for fixed $t\in T$, $F\mapsto b(t,x;F)$ is continuous
on $\mathscr{F}$ in $\rho$, $b^{N}$ is the approximate solution
to the forward problem given (\ref{eq:9N})-(\ref{eq:9Nend}), $J^{N}$
is the approximation given in (\ref{eq:ApproxInv}), and $Q_{D}$
a countable dense subset of $Q$ as defined in (\ref{eq:QD}). Moreover,
suppose that $\mathscr{F}_{N}^{*ML}(\mathbf{n}^{d})$ is the set of
minimizers for $J^{N}(F;\mathbf{n}^{d})$ over $F\in\mathscr{F}^{ML}$
corresponding to the data $\mathbf{n}^{d}$. Similarly, suppose that
$\mathscr{F}^{*}(\mathbf{n}^{d})$ is the set of minimizers for $J(F;\mathbf{n}^{d})$
over $F\in\mathscr{F}$ corresponding to the data $\mathbf{n}^{d}$.
Then, $d_{H}(\mathscr{F}_{N}^{*ML}(\mathbf{n}^{d}),\,F_{0})\to0$
as $N,\,M,\,L,\,N_{t},\,N_{x}\to\infty$.\end{thm}
\begin{proof}
Observe that a simple triangle inequality yields 
\[
d_{H}\left(\mathscr{F}_{N}^{*ML}(\mathbf{n}^{d}),\,F_{0}\right)\le d_{H}\left(\mathscr{F}_{N}^{*ML}(\mathbf{n}^{d}),\,\mathscr{F}^{*ML}(\mathbf{n}^{d})\right)+d_{H}\left(\mathscr{F}^{*ML}(\mathbf{n}^{d}),\,F_{0}\right)\,.
\]
Therefore, combining the arguments of Theorem~\ref{thm:Orig Cont dep on data}
and Theorem~\ref{thm:Exist/conv of orig =000026 approx mins}, we
readily obtain that $d_{H}\left(\mathscr{F}_{N}^{*ML}(\mathbf{n}^{d}),\,F_{0}\right)$
converges to zero as $N,\,M,\,L,\,N_{t},\,N_{x}\to\infty$.
\end{proof}
With the results of these two theorems, we can claim that both there
exists a solution to the inverse problem and it is continuously dependent
on the given data. We have established method stability under approximation
of the state space and parameter space of our inverse problem. Therefore
we can conclude general well-posedness of the inverse problem.

\section{\label{sec:Example-Illustration}Example Illustration}

The particular model we study here is the size-structured flocculation
dynamics of the microorganisms in suspension and is given by the following
integro-differential equation
\begin{eqnarray}
\qquad b_{t} & = & \mathcal{A}[b]+\mathcal{B}[b]+\mathcal{R}[b],\label{eq:9}\\
b(t,\,0) & = & 0,\\
b(0,\,x) & = & b_{0}(x),\label{eq:9end}
\end{eqnarray}

\noindent where $b(t,\,x)\,dx$ is the number of aggregates with volumes
in $[x,\,x+dx]$ at time $t$, and $\mathcal{A}$, $\mathcal{B}$
and $\mathcal{R}$ are the aggregation, breakage (fragmentation) and
removal operators, respectively. We consider $x\in Q=[0,\,\overline{x}]$,
where $\overline{x}$ is the maximum floc volume and $t\in T=[0,\,t_{f}]$,
$t_{f}<\infty$. The aggregation, fragmentation and removal functions
are defined by:

\begin{eqnarray}
\mathcal{A}[p](t,\,x) & := & \frac{1}{2}\int_{0}^{x}k_{a}(x-y,\,y)p(t,\,x-y)p(t,\,y)\,dy\nonumber \\
 & \quad & -p(t,\,x)\int_{0}^{\overline{x}}k_{a}(x,\,y)p(t,\,y)\,dy\,,\label{eq:Aggregation}
\end{eqnarray}
 
\begin{eqnarray}
\mathcal{B}[p](t,\,x) & := & \int_{x}^{\overline{x}}\Gamma(x;\,y)k_{f}(y)p(t,\,y)\,dy-\frac{1}{2}k_{f}(x)p(t,\,x)\label{eq:Breakage}
\end{eqnarray}
and 
\begin{equation}
\mathcal{R}[p](t,\,x):=-\mu(x)p(t,\,x)\,.\label{eq:Removal}
\end{equation}

\noindent where $k_{a}(x,\,y)$ is the aggregation kernel, describing
the rate at which flocs of volume $x$ and $y$ combine to form a
floc of volume $x+y$. The aggregation kernel is symmetric function
and $k_{a}(x,\,y)=0$ for $x+y>\overline{x}$. The fragmentation kernel
$k_{f}(x)$ describes the rate at which a floc of volume $x$ fragments.
The function $\Gamma(x,\,y)$ is the post-fragmentation probability
density, for the conditional probability of producing a daughter floc
of size $x$ from a mother floc of size $y$. This probability density
is used to characterize the stochastic nature of floc fragmentation
(e.g., see the discussions in \cite{Bortz2008,HanEtal2003AICHEJ,BablerEtal2008jfm,ByrneEtal2011}).

In \cite{ByrneEtal2011}, we proposed a model for bacterial floc breakage
based upon hydrodynamic arguments and predicted a post fragmentation
density $\Gamma$. The eventual goal (and the topic for a future paper)
is to unify the theoretical results (in this work and in \cite{Bortz2008,ByrneEtal2011})
with experimental evidence to validate (or refute) our proposed fragmentation
model. We now consider the application of this framework to the system
in (\ref{eq:9})-(\ref{eq:9end}). For fixed $t\in I,\,$$b(t,\cdot)\in H$,
$F\in\mathscr{F}$, consider the right side of (\ref{eq:9}), represented
by (\ref{eq:AbstractEvEqn}), 
\[
g(b,\,F)=\mathcal{A}[b]+\mathcal{B}[b;\,F]+\mathcal{R}[b]\,.
\]
 To show that $g$ satisfies the locally Lipschitz property of Condition~\ref{con:gLipschitz},
we need the following two lemmas.
\begin{lem}
Suppose that $k_{f},\,\mu\in L^{\infty}(Q)$ and $k_{a},\,\Gamma\in L^{\infty}\left(Q\times Q\right)$.
The evolution equation (\ref{eq:9})-(\ref{eq:9end}) is well-posed
on $H=L^{1}\left(Q,\,\mathbb{R}^{+}\right)$ and for any compact set
$T=[0,\,t_{f}]$ and $b_{0}\ge0$, the classical solution of (\ref{eq:9})-(\ref{eq:9end})
satisfies
\[
C_{0}=\sup_{t\in T,\,x\in Q}\left|b(t,\,x)\right|<\infty\,.
\]
Furthermore, the operator $\mathcal{A}+\mathcal{R}$ is locally Lipschitz
\[
\left\Vert \mathcal{A}[b]+\mathcal{R}[b]-\mathcal{A}[\tilde{b}]-\mathcal{R}[\tilde{b}]\right\Vert \le C_{1}\left\Vert b-\tilde{b}\right\Vert 
\]
where $C_{1}=3C_{0}\left\Vert k_{a}\right\Vert _{\infty}+\left\Vert \mu\right\Vert _{\infty}$.\end{lem}
\begin{proof}
For the proof of the first part we refer readers to \cite[\S 3]{Bortz2008}.
To show that $\mathcal{A}+\mathcal{R}$ is locally Lipschitz, first
observe that 
\begin{eqnarray*}
\left\Vert \mathcal{A}[b]-\mathcal{A}[\tilde{b}]\right\Vert  & \le & \frac{1}{2}\int_{Q}\left|\int_{0}^{x}k_{a}(x-y,\,y)b(x-y)b(y)\,dy\right.\\
 &  & \left.-\int_{0}^{x}k_{a}(x-y,\,y)\tilde{b}(x-y)\tilde{b}(y)\,dy\right|dx\\
 & \quad & +\int_{Q}\left|b(x)\int_{Q}k_{a}(x,\,y)b(y)\,dy-\tilde{b}(x)\int_{Q}k_{a}(x,\,y)\tilde{b}(y)\,dy\right|\,dx\\
 & \le & \left\Vert k_{a}\right\Vert _{\infty}\left[\frac{1}{2}\int_{Q}\left|\int_{0}^{x}b(x-y)\left(b(y)-\tilde{b}(y)\right)\,dy\right|\,dx\right.\\
 & \quad & +\frac{1}{2}\int_{Q}\left|\int_{0}^{x}\tilde{b}(y)\left(b(x-y)-\tilde{b}(x-y)\right)\,dy\right|\,dx\\
 & \quad & +\int_{Q}\left|b(x)\int_{Q}\left(b(y)-\tilde{b}(y)\right)\,dy\right|\,dx\\
 & \quad & \left.+\int_{Q}\left|\tilde{b}(x)\int_{Q}\left(b(y)-\tilde{b}(y)\right)\,dy\right|\,dx\right]\,.
\end{eqnarray*}
At this point applying Young's inequality \cite[Theorem 2.24]{Adams2003}
for the first two integrals yields the desired result
\begin{eqnarray*}
\left\Vert \mathcal{A}[b]-\mathcal{A}[\tilde{b}]\right\Vert  & \le & \left\Vert k_{a}\right\Vert _{\infty}\left[\frac{1}{2}\left\Vert b\right\Vert \left\Vert b-\tilde{b}\right\Vert +\frac{1}{2}\left\Vert \tilde{b}\right\Vert \left\Vert b-\tilde{b}\right\Vert \right.\\
 &  & \left.+\left\Vert b\right\Vert \left\Vert b-\tilde{b}\right\Vert +\left\Vert \tilde{b}\right\Vert \left\Vert b-\tilde{b}\right\Vert \right]\\
 & \le & 3C_{0}\left\Vert k_{a}\right\Vert _{\infty}\left\Vert b-\tilde{b}\right\Vert \,.
\end{eqnarray*}

\end{proof}
The above lemma establishes that the classical solution of (\ref{eq:9})-(\ref{eq:9end})
is bounded on $T\times Q$. Moreover, since the space of Riemann integrable
functions are dense on $L^{1}(Q,\,\mathbb{R}^{+})$, we tacitly assume
that the classical solution is also Riemann integrable. Therefore,
the evolution equation (\ref{eq:9})-(\ref{eq:9end}) satisfies consistency
conditions of Theorem \ref{thm:Consistency of the estimator}, and
thus the inverse problem defined in (\ref{eq:OrigInv}) is well-posed
for this particular application.
\begin{lem}
The fragmentation operator $\mathcal{B}$ satisfies the locally Lipschitz
property of Condition~\ref{con:gLipschitz}.\end{lem}
\begin{proof}
Examining the fragmentation term, we find 
\begin{eqnarray*}
\left\Vert \mathcal{B}(b,F)-\mathcal{B}(\tilde{b},\tilde{F})\right\Vert  & \leq & \left\Vert \frac{1}{2}k_{f}(x)\left(\tilde{b}(t,\,x)-b(t,\,x)\right)\right\Vert \\
 &  & +\left\Vert \int_{x}^{\overline{x}}k_{f}(y)\left(b(t,\,y)\,\Gamma(x,y)-\tilde{b}(t,\,y)\,\tilde{\Gamma}(x,y)\right)dy\right\Vert \\
 & \leq & \frac{1}{2}C_{\text{frag}}\left\Vert b-\tilde{b}\right\Vert +K\left\Vert \int_{Q}b(t,\,y)\left(\Gamma(x,y)-\tilde{\Gamma}(x,y)\right)dy\right\Vert \\
 &  & +C_{\text{frag}}\left\Vert \int_{Q}\left(b(t,\,y)-\tilde{b}(t,\,y)\right)\tilde{\Gamma}(x,y)\,dy\right\Vert 
\end{eqnarray*}

\noindent where $C_{\text{frag}}=\left\Vert k_{f}\right\Vert _{\infty}$.
The second term on the right hand side becomes 
\begin{eqnarray*}
C_{\text{frag}}\left\Vert \int_{Q}b(t,\,y)\left(\Gamma(x,y)-\tilde{\Gamma}(x,y)\right)dy\right\Vert  & \leq & C_{\text{frag}}\int_{Q}\int_{Q}\left|b(t,\,y)\right|\left|\Gamma(x,y)-\tilde{\Gamma}(x,y)\right|dy\,dx\\
 & \leq & C_{\text{frag}}\int_{Q}\left|b(t,\,y)\right|\int_{Q}\left|\Gamma(x,y)dx-\tilde{\Gamma}(x,y)dx\right|dy\\
 & \leq & C_{\text{frag}}\int_{Q}\left|b(t,\,y)\right|\left(\int_{Q}\left|dF_{y}-d\tilde{F}_{y}\right|\right)dy\\
 & \le & C_{\text{frag}}\sup_{y\in Q}\int_{Q}\left|dF_{y}-d\tilde{F}_{y}\right|\int_{Q}\left|b(t,\,y)\right|dy\\
 & \le & C_{\text{frag}}\overline{x}C_{0}\sup_{y\in Q}\int_{Q}\left|dF_{y}-d\tilde{F}_{y}\right|
\end{eqnarray*}
 Since $\int_{Q}\left|dF_{y}-d\tilde{F}_{y}\right|\to0$ is equivalent
to $\rho_{Proh}(F_{y},\tilde{F}_{y})\to0$ , we know that 
\[
\sup_{y\in Q}\int_{Q}\left|dF_{y}-d\tilde{F}_{y}\right|\to0\quad\text{as}\quad\rho(F,\tilde{F})\to0.
\]
 Therefore, 
\[
C_{\text{frag}}\left\Vert \int_{Q}b(y)\left(\Gamma(x,y)-\tilde{\Gamma}(x,y)\right)dy\right\Vert \to0\quad\text{as}\quad\rho(F,\,\tilde{F})\to0.
\]
 Similar analysis for the third term leads to the bound 
\[
C_{\text{frag}}\left\Vert \int_{Q}\left(b(y)-\tilde{b}(t,y)\right)\tilde{\Gamma}(x,y)\,dy\right\Vert \leq C_{\text{frag}}\overline{x}\left\Vert \Gamma\right\Vert _{\infty}\left\Vert b-\tilde{b}\right\Vert .
\]
 Combining these results we find the overall fragmentation term can
be bounded by 
\[
\left\Vert \mathcal{B}(b,\phi)-\mathcal{B}(\tilde{b},\tilde{\phi})\right\Vert \leq C_{\text{frag}}\left(\frac{1}{2}+\overline{x}\left\Vert \Gamma\right\Vert _{\infty}\right)\left\Vert b-\tilde{b}\right\Vert +\mathscr{T}(F,\,\tilde{F}).
\]
\end{proof}
\begin{claim}
\label{claim:gLipschitz}The function $g$ satisfies the locally Lipschitz
property of Condition~\ref{con:gLipschitz}.\end{claim}
\begin{proof}
Consider 
\begin{eqnarray*}
\left\Vert g(b,F)-g(\tilde{b},\tilde{F})\right\Vert  & = & \left\Vert \mathcal{A}[b]-\mathcal{A}[\tilde{b}]+\mathcal{B}[b;F]-\mathcal{B}[\tilde{b};\tilde{F}]+\mathcal{R}[b]-\mathcal{R}[\tilde{b}]\right\Vert \\
 & \leq & \left\Vert \mathcal{A}[b]-\mathcal{A}[\tilde{b}]\right\Vert +\left\Vert \mathcal{B}[b;F]-\mathcal{B}[\tilde{b};\tilde{F}]\right\Vert +\left\Vert \mathcal{R}[b]-\mathcal{R}[\tilde{b}]\right\Vert \,.
\end{eqnarray*}

\noindent Using the Lipschitz constants from the fragmentation and
aggregation terms, 
\[
\left\Vert g(b,\phi)-g(\tilde{b},\tilde{\phi})\right\Vert \leq C\left\Vert b-\tilde{b}\right\Vert +\mathscr{T}(F,\,\tilde{F})
\]
 where $C=C_{\text{frag}}\left(\frac{1}{2}+\overline{x}\left\Vert \Gamma\right\Vert _{\infty}\right)+C_{1}$.
\end{proof}
\noindent Therefore, since the function $g$ satisfies Condition~\ref{con:gLipschitz},
we can conclude well-posedness of the inverse problem for identifying
the post-fragmentation probability density, $\Gamma(x,y)$, found
in the model for flocculation dynamics of bacterial aggregates described
in (\ref{eq:9})-(\ref{eq:9end}).

\subsection{Numerical Implementation}

We first form an approximation to $H$. We define basis elements
\[
\beta_{i}^{N}(x)=\left\{ \begin{array}{cc}
1;\, & x_{i-1}^{N}\leq x\leq x_{i}^{N}\,;\,i=1,\ldots,N\\
0;\, & \mbox{otherwise}
\end{array}\right.
\]
for positive integer $N$ and $\{x_{i}^{N}\}_{i=0}^{N}$ a uniform
partition of $[0,\overline{x}]=[x_{0}^{N},x_{N}^{N}]$, and $\Delta x=x_{j}^{N}-x_{j-1}^{N}$
for all $j$. The $\beta^{N}$ functions form an orthogonal basis
for the approximate solution space 
\[
H^{N}=\left\{ h\in H\;|\;h=\sum_{i=1}^{N}\alpha_{i}\beta_{i}^{N},\;\alpha_{i}\in\mathbb{R}\right\} ,
\]
 and accordingly, we define the orthogonal projections $\pi^{N}:\,H\mapsto H^{N}$
\[
\pi^{N}h=\sum_{j=1}^{N}\,\alpha_{j}\beta_{j}^{N},\qquad\text{where}\:\alpha_{j}=\frac{1}{\Delta x}\,\int_{x_{j-1}^{N}}^{x_{j}^{N}}\,h(x)\,dx.
\]
 Thus our approximating formulations of (\ref{eq:9}), (\ref{eq:9end})
becomes the following system of $N$ ODEs for $b^{N}\in H^{N}$ and
$F\in\mathscr{F}$:

\begin{eqnarray}
b_{t}^{N} & = & \pi^{N}\left(\mathcal{A}[b^{N}]+\mathcal{B}[b^{N};F]+\mathcal{R}[b^{N}]\right),\label{eq:9N}\\
b^{N}(0,x) & = & \pi^{N}b_{0}(x)\,,\label{eq:9Nend}
\end{eqnarray}
where {\footnotesize{}
\begin{eqnarray*}
\pi^{N}\mathcal{A}[b^{N}] & = & \left(\begin{array}{c}
-\alpha_{1}\sum_{j=1}^{N-1}k_{a}(x_{1},\,x_{j})\alpha_{j}\Delta x\\
\frac{1}{2}k_{a}(x_{1},x_{1})\alpha_{1}\alpha_{1}\Delta x-\alpha_{2}\sum_{j=1}^{N-2}k_{a}(x_{2},\,x_{j})\alpha_{j}\Delta x\\
\vdots\\
\frac{1}{2}\sum_{j=1}^{N-2}k_{a}(x_{j},x_{N-1-j})\alpha_{j}\alpha_{N-1-j}\Delta x-\alpha_{N-1}k_{a}(x_{N-1},\,x_{1})\alpha_{1}\Delta x\\
\frac{1}{2}\sum_{j=1}^{N-1}k_{a}(x_{j},x_{N-j})\alpha_{j}\alpha_{N-j}\Delta x
\end{array}\right)
\end{eqnarray*}
}and {\footnotesize{}
\begin{eqnarray*}
\pi^{N}\left(\mathcal{B}[b^{N};F]+\mathcal{R}[b^{N}]\right) & = & \left(\begin{array}{c}
\sum_{j=2}^{N}\Gamma(x_{1};x_{j})k_{f}(x_{j})\alpha_{j}\Delta x-\frac{1}{2}k_{f}(x_{1})\alpha_{1}-\mu(x_{1})\alpha_{1}\\
\sum_{j=3}^{N}\Gamma(x_{2};x_{j})k_{f}(x_{j})\alpha_{j}\Delta x-\frac{1}{2}k_{f}(x_{2})\alpha_{2}-\mu(x_{2})\alpha_{2}\\
\vdots\\
\Gamma(x_{N-1};x_{N})k_{f}(x_{N})\alpha_{N}\Delta x-\frac{1}{2}k_{f}(x_{N-1})\alpha_{N-1}-\mu(x_{N-1})\alpha_{N-1}\\
-\frac{1}{2}k_{f}(x_{N})\alpha_{N}-\mu(x_{N})\alpha_{N}
\end{array}\right)\,.
\end{eqnarray*}
}In the following lemma we show that the numerical scheme satisfies
Condition \ref{cond:gNLipschitz}. 
\begin{claim}
The function $g^{N}:H^{N}\times\mathscr{F}\to H^{N}$ as defined by
\[
g^{N}(b^{N},F)=\mathcal{A}[b^{N}]+\mathcal{B}[b^{N};F]+\mathcal{R}[b^{N}]
\]
satisfies the Lipschitz-type property in Condition~\ref{cond:gNLipschitz}.\end{claim}
\begin{proof}
\noindent We consider the integrand 
\[
\left\Vert \pi^{N}\left(\mathcal{A}[b^{N}(t,x;\tilde{F})]+\mathcal{B}[b^{N}(t,x;\tilde{F})]-\mathcal{A}[b^{N}(t,x;F)]-\mathcal{B}[b^{N}(t,x;F)]\right)\right\Vert \,,
\]
and note that 
\[
\begin{array}{cc}
\leq & \left\Vert \pi^{N}\right\Vert \left(\left\Vert \mathcal{A}[b(t,x;\tilde{F})]-\mathcal{A}[b^{N}(t,x;F)]\right\Vert \right.\\
 & \qquad\left.+\left\Vert \mathcal{B}[b^{N}(t,x;\tilde{F})]-\mathcal{B}[b^{N}(t,x;F)]\right\Vert \right)\,.
\end{array}
\]
The induced ($L^{1}$-) norm on the projection operator will not be
an issue as 
\begin{eqnarray*}
\left\Vert \pi^{N}\right\Vert  & = & \sup_{h\in H,\left\Vert h\right\Vert =1}\left\Vert \pi^{N}h\right\Vert \\
 & = & \sup_{h\in H,\left\Vert h\right\Vert =1}\left\Vert \sum_{j=1}^{N}\frac{\beta_{j}^{N}(\cdot)}{\Delta x}\int_{x_{j-1}^{N}}^{x_{j}^{N}}h(x)dx\right\Vert \\
 & = & 1\,.
\end{eqnarray*}
As illustrated in the proof of Claim \ref{claim:gLipschitz} above,
the bounding constants for $\mathcal{A}$ and $\mathcal{B}$ are $C_{\text{agg}}$
and $\frac{3}{2}C_{\text{frag}}$, respectively. 

\noindent Combining these results, 
\[
\left\Vert b^{N}(t,x;\tilde{F})-b^{N}(t,x;F)\right\Vert \leq C_{N}\left\Vert b^{N}(t,x;\tilde{F})-b^{N}(t,x;F)\right\Vert +\mathcal{T}^{N}(\tilde{F},F)
\]
where $\mathcal{T}^{N}(\tilde{F},F)=\int_{0}^{t_{f}}\pi^{N}\mathscr{T}(\tilde{F},F)ds$,
independent of $t$, and $C_{N}=\left\Vert k_{f}\right\Vert _{\infty}\left(\frac{1}{2}+\overline{x}\left\Vert \Gamma\right\Vert _{\infty}\right)+3C_{0}\left\Vert k_{a}\right\Vert _{\infty}+\left\Vert \mu\right\Vert _{\infty}\,.$ \end{proof}
\begin{cor}
The semi-discrete solutions to (\ref{eq:9N}) converge uniformly in
norm to the unique solution of (\ref{eq:9}) on a bounded time interval
as $N\to\infty$.\end{cor}
\begin{proof}
From results in \cite{Bortz2008}, we can obtain semi-discrete solutions
$b^{N}$ to the forward problem that converge uniformly in norm to
the unique solution of (\ref{eq:9})-(\ref{eq:9end}) on a bounded
time interval as $N\to\infty$. 

For fixed $N$, we rewrite (\ref{eq:9N}) in integral form and consider 

\begin{eqnarray*}
\left\Vert b^{N}(t,x;F)-\pi^{N}b(t,x;F)\right\Vert  & \leq & \int_{0}^{t}\left\Vert \pi^{N}\left(\mathcal{R}[b^{N}(s,x;F)]-\mathcal{R}[b(s,x;F)]\right)\right\Vert ds\\
 &  & +\int_{0}^{t}\left\Vert \pi^{N}\left(\mathcal{A}[b^{N}(s,x;F)]+\mathcal{B}[b(s,x;F)]\right.\right.\\
 &  & \left.\left.-\mathcal{A}[b^{N}(s,x;F)]-\mathcal{B}[b(s,x;F]\right)\right\Vert ds.
\end{eqnarray*}
for $t\in T$. Our strategy is to use the fact that our discretized
version of $g$ is locally Lipschitz and then apply Gronwall's inequality.
We refer readers to \cite{Ackleh1997,Bortz2008} for the detailed
discussion about the convergence of the numerical scheme. 
\end{proof}

\subsection{Results}

\noindent \begin{flushleft}
As an initial investigation into the utility of this approach, we
applied the method to the problem of flocculation dynamics. In \cite{ByrneEtal2011},
we have shown that the post-fragmentation density function greatly
depends on the parent floc size. In particular, we found that the
resulting post-fragmentation density for large parent flocs resembles
a Beta distribution with $\alpha=\beta=0.5$. For small flocs, however,
the resulting density resembles a Beta distribution with $\alpha=\beta=2$.
Towards this end, we applied the framework presented in this paper
to these two different post-fragmentation functions. In Figure \ref{fig:evidence},
data were generated from the forward problem by assuming a post-fragmentation
density function, 
\[
\Gamma_{\text{true}}(x,y)=\mathbbm{1}_{[0,\,y]}(x)\frac{6x(y-x)}{y^{3}}\,.
\]
 Similarly, in Figure \ref{fig:evidence-1}, data were generated with
a post-fragmentation density function,
\[
\Gamma_{\text{true}}(x,y)=\mathbbm{1}_{[0,\,y]}(x)\frac{1}{\pi\sqrt{x(y-x)}}\,.
\]
\begin{figure}[ht]
\begin{centering}
\subfloat[\label{fig:ftrue_beta}]{\protect\includegraphics[scale=0.4]{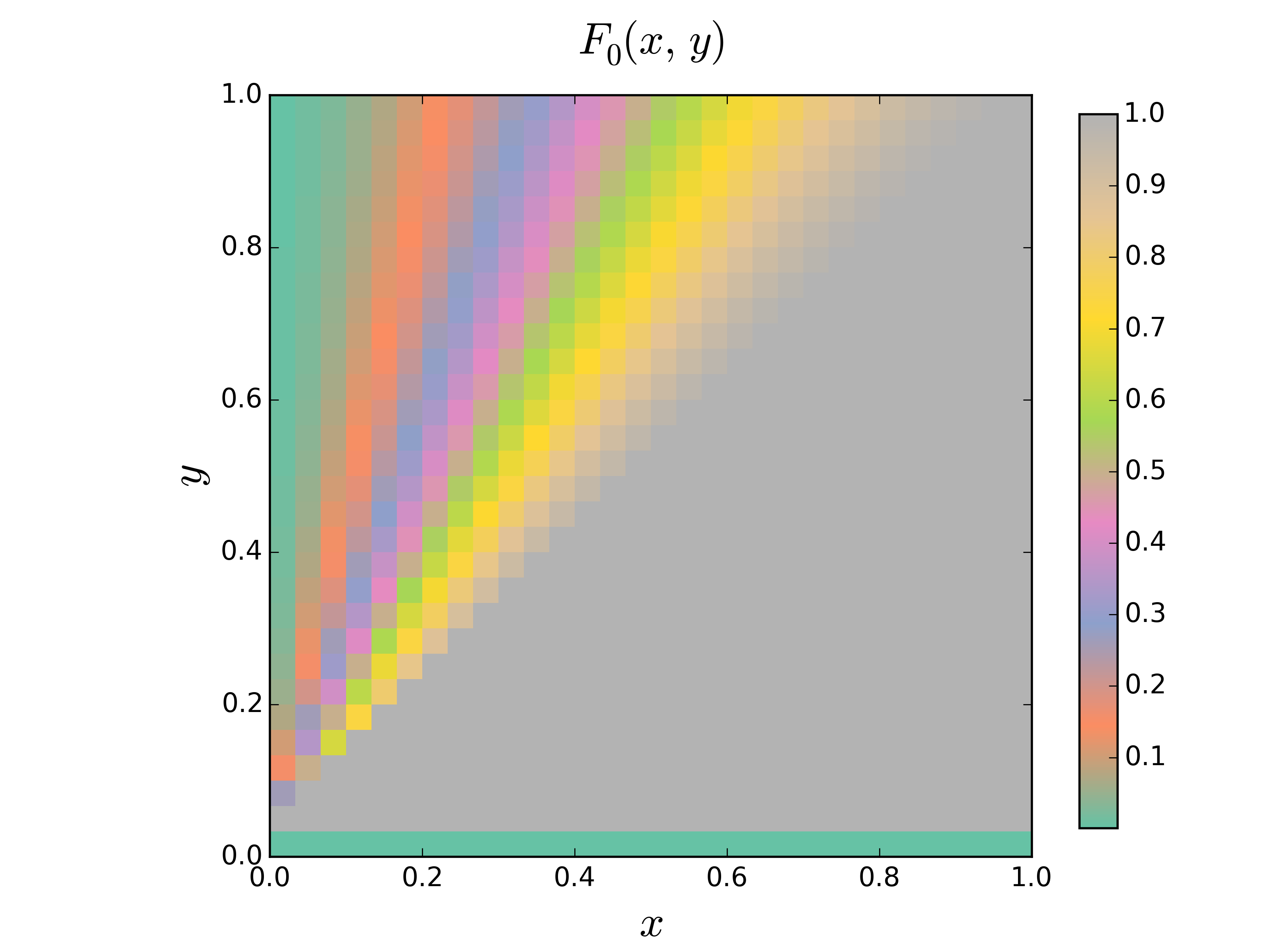}}~\subfloat[\label{fig:f_fit_beta}]{\protect\includegraphics[scale=0.4]{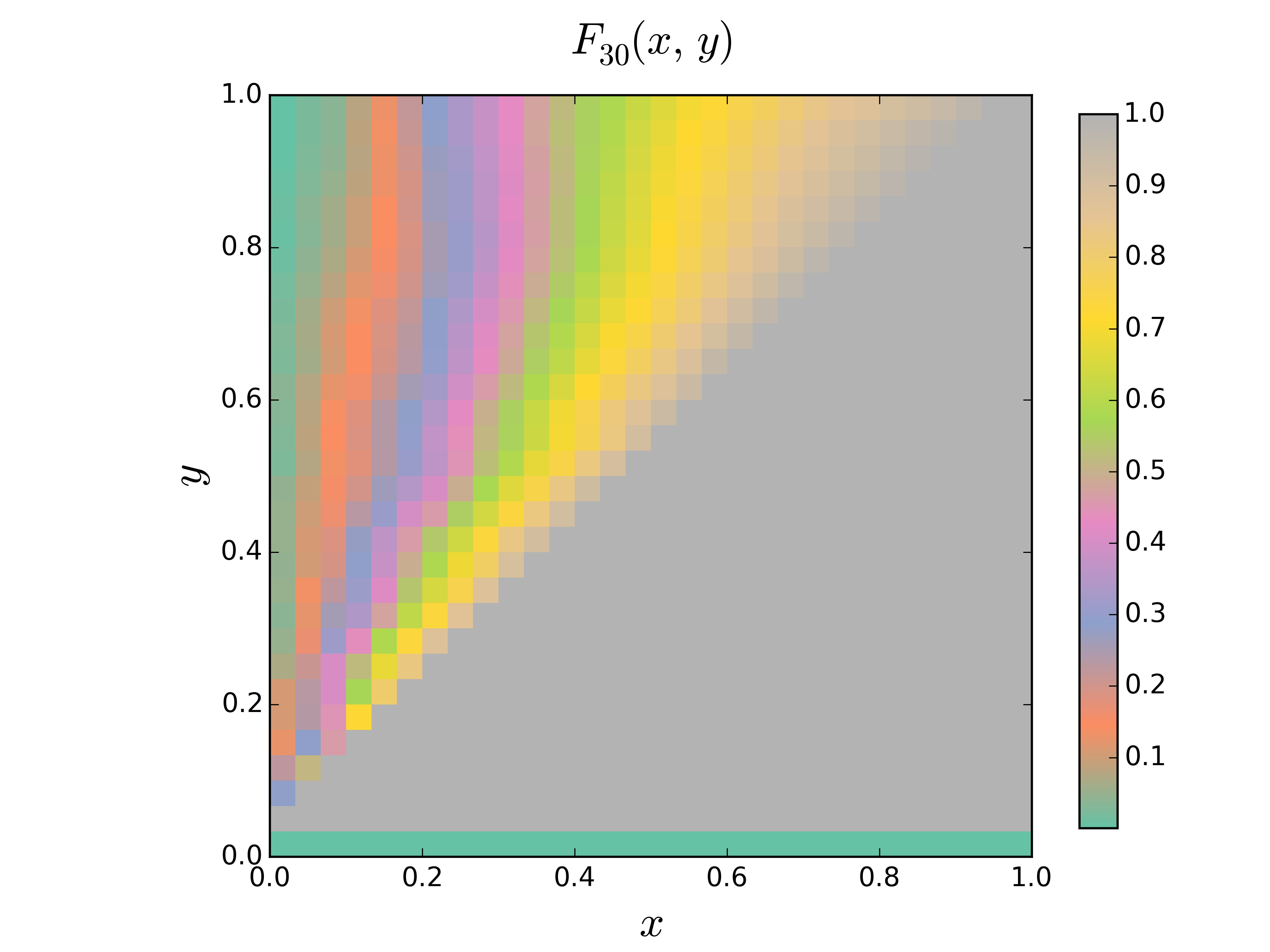}

}\\
\subfloat[\label{fig:abs_error_beta}]{\protect\includegraphics[scale=0.4]{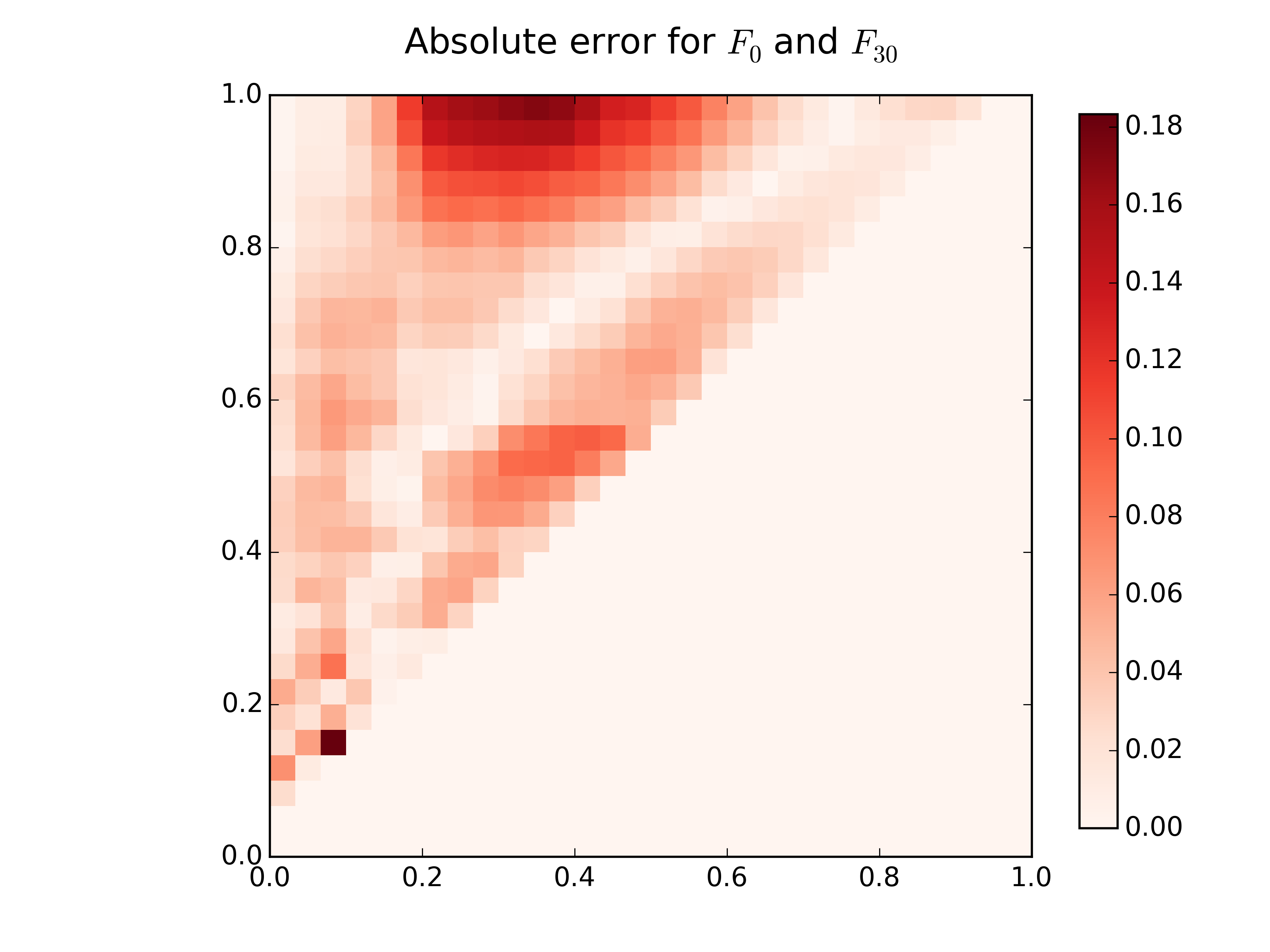}}~\subfloat[\label{fig:error_beta_N}]{\protect\includegraphics[scale=0.4]{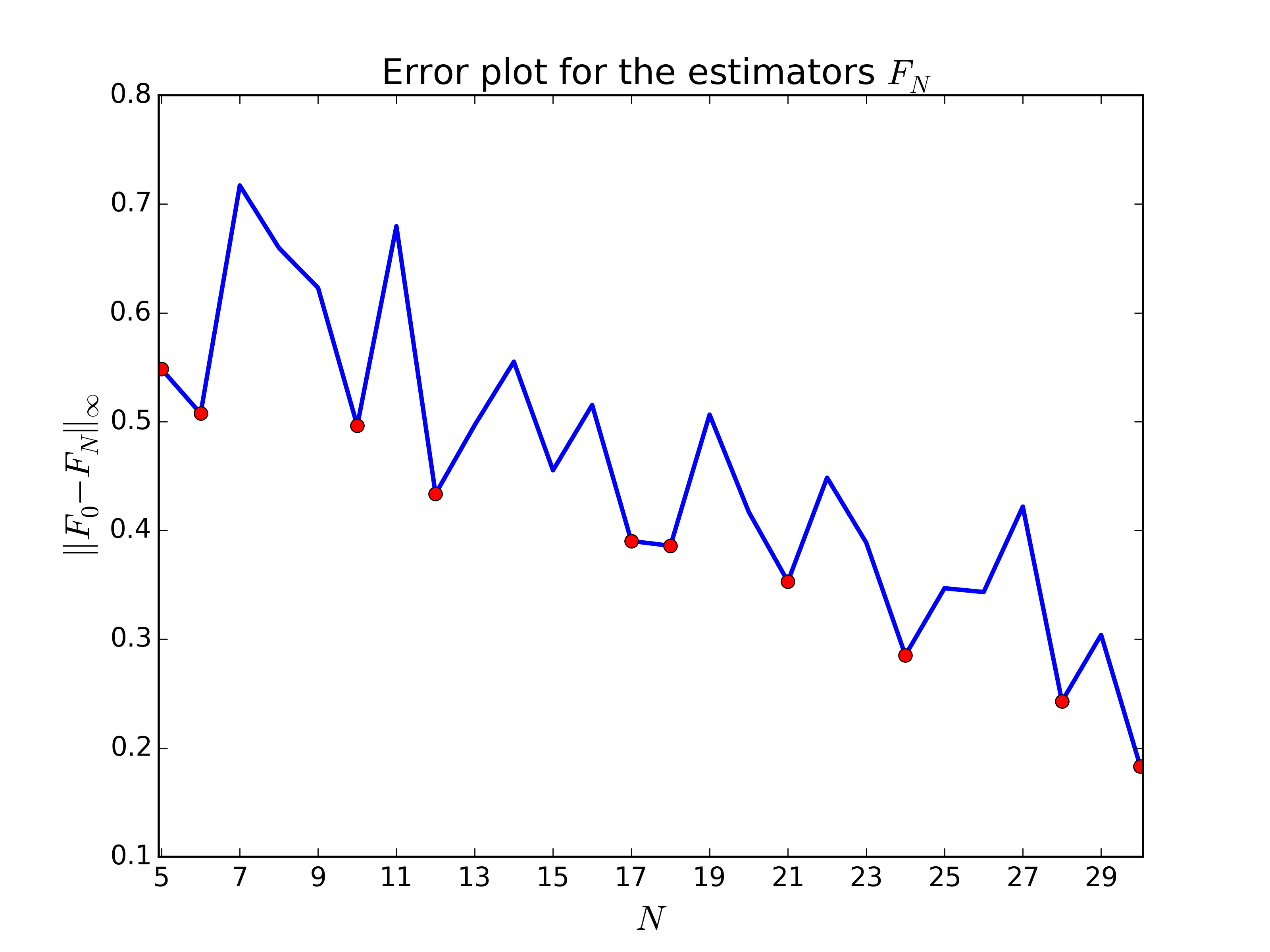}

}
\par\end{centering}

\centering{}\protect\caption{\label{fig:evidence}Simulation results for pseudo-data generated
with Beta distribution with $\alpha=\beta=2$ (a) Post-fragmentation
density used to generate test data (\emph{true} post-fragmentation
density). (b) Results of the optimization scheme based on these data
and an initial density function uniform in $x$. (c) Absolute error
for true probability measure $F_{0}$ and the approximate estimator
$F_{30}$. (d) Error plots for the sequence of estimators $\left\{ F_{N}\right\} _{N=5}^{30}$
. Solid red dots indicate the convergent subsequence of the estimators. }
\end{figure}

\par\end{flushleft}

\noindent \begin{flushleft}
\begin{figure}[ht]
\begin{centering}
\subfloat[\label{fig:ftrue_arcsin}]{\protect\includegraphics[scale=0.4]{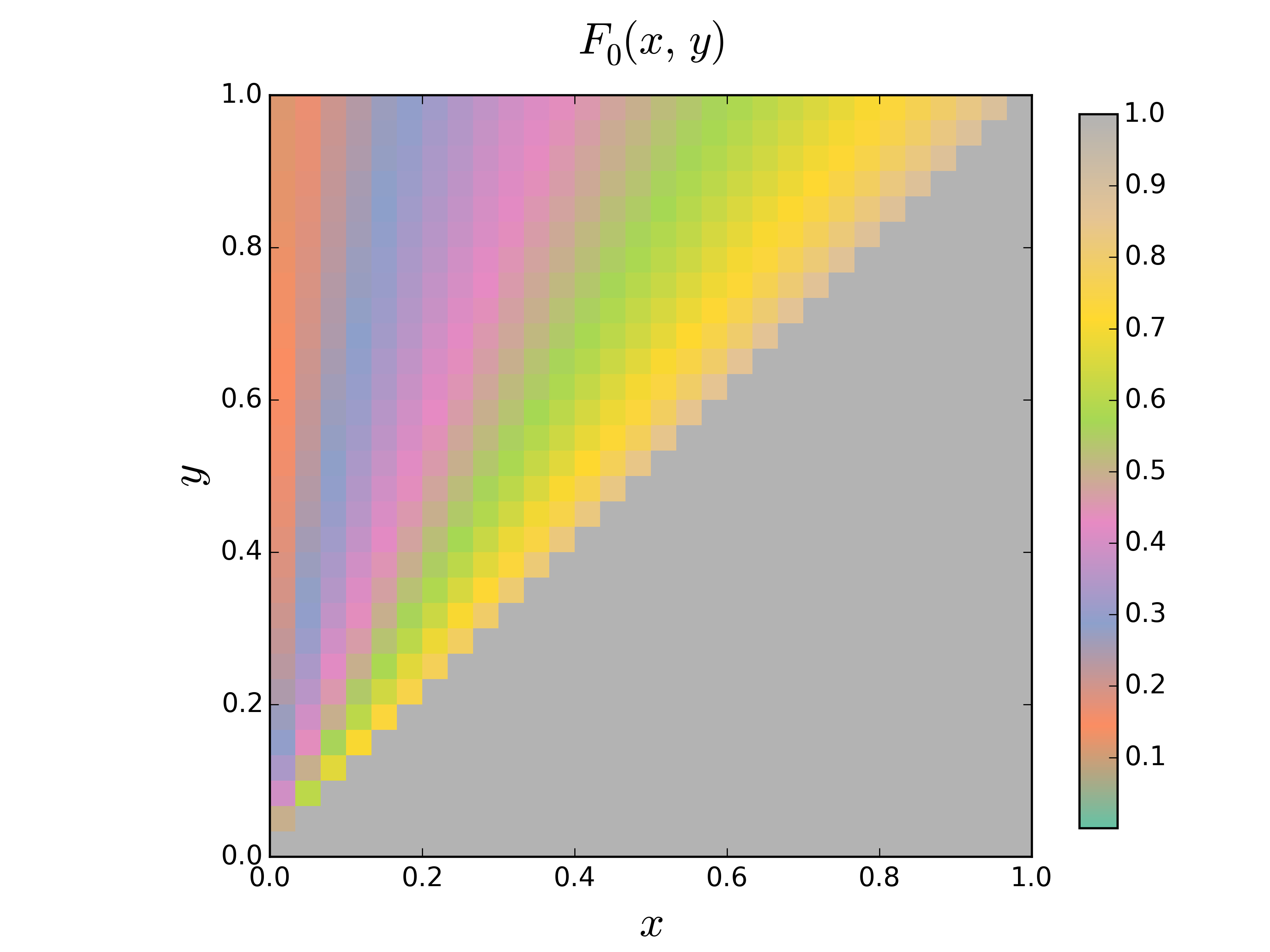}}~\subfloat[\label{fig:f_fit_arcsin}]{\protect\includegraphics[scale=0.4]{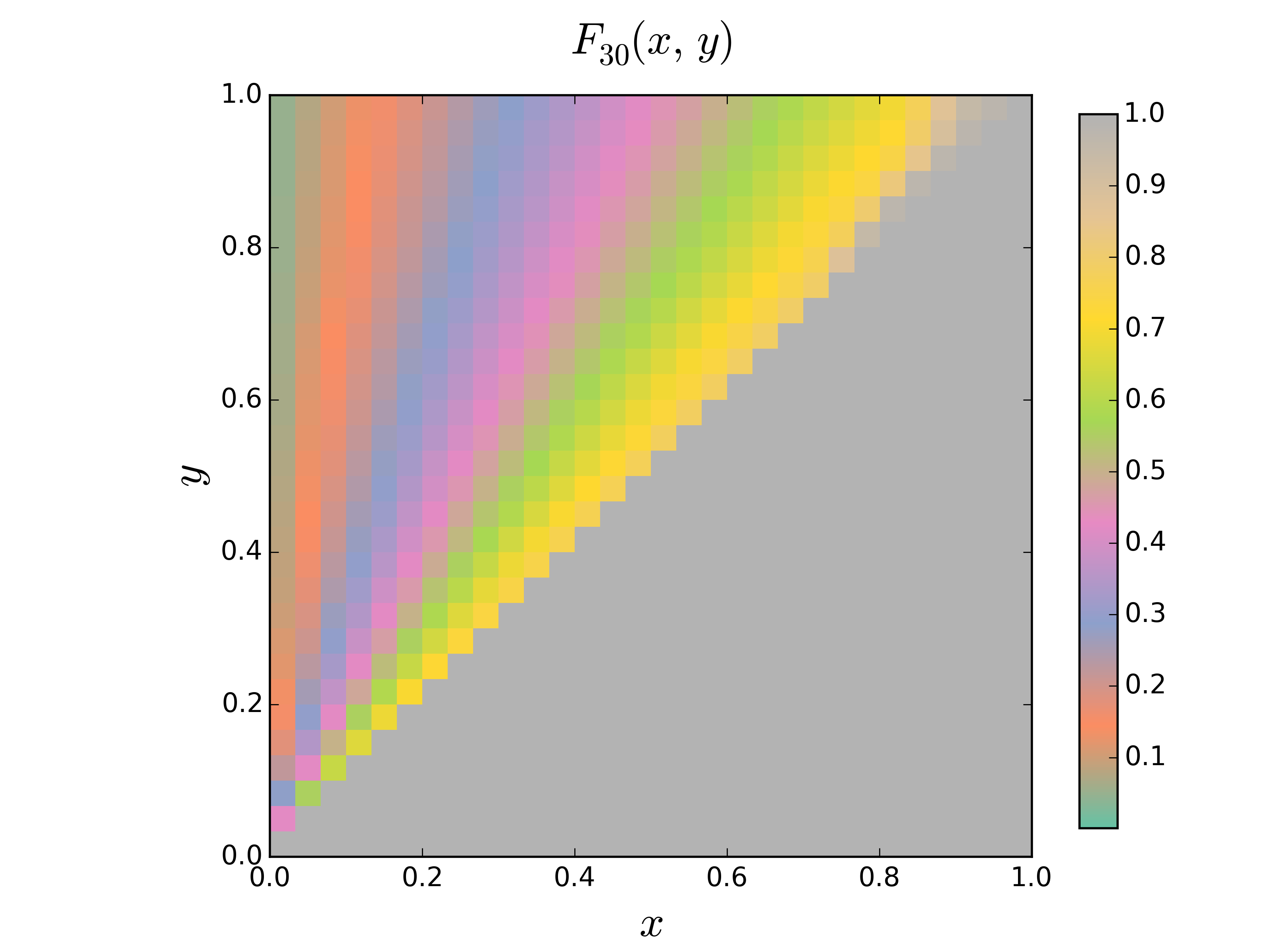}

}\\
\subfloat[\label{fig:abs_error_arcsin}]{\protect\includegraphics[scale=0.4]{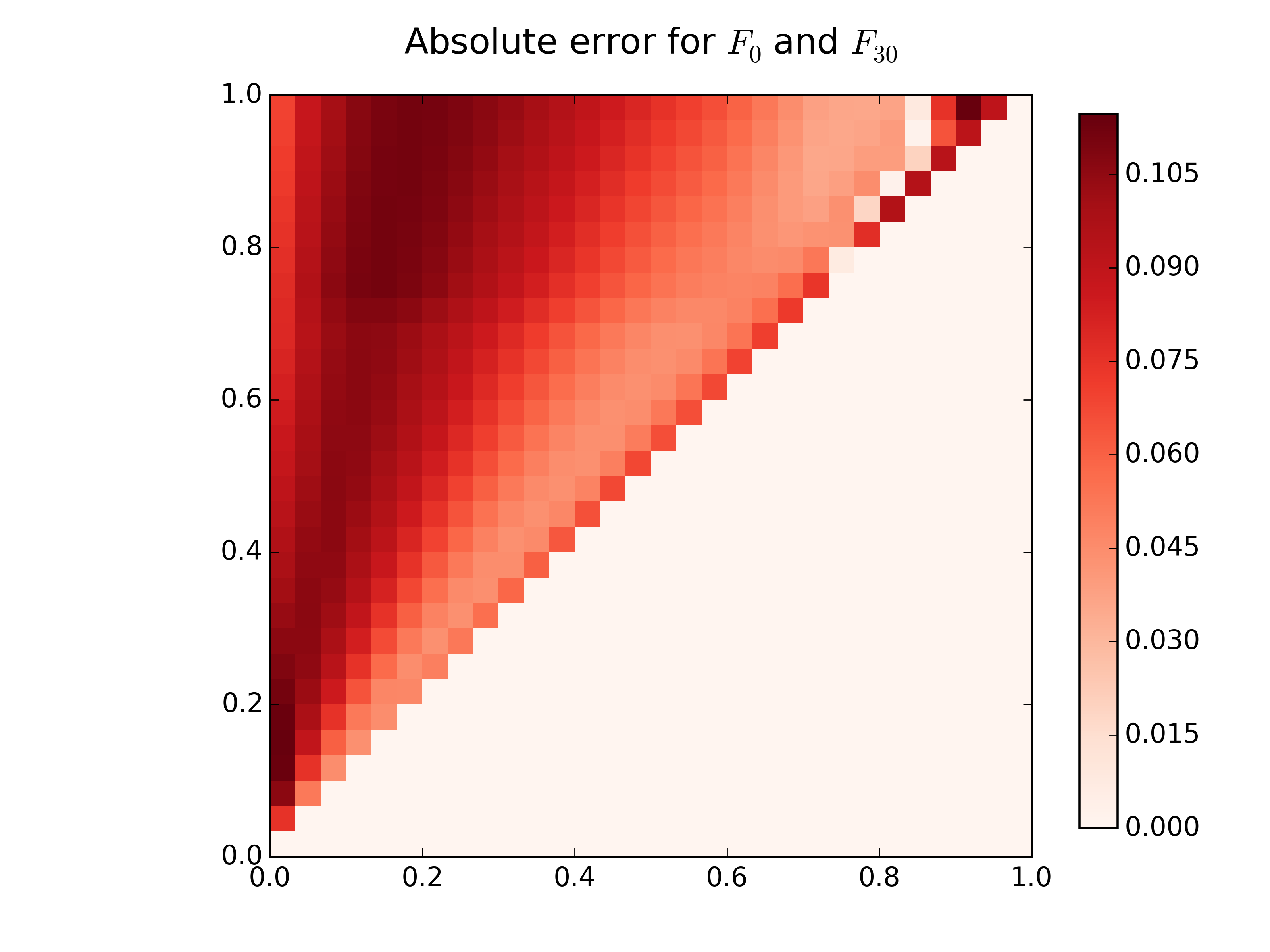}}~\subfloat[\label{fig:error_arcsin_N}]{\protect\includegraphics[scale=0.4]{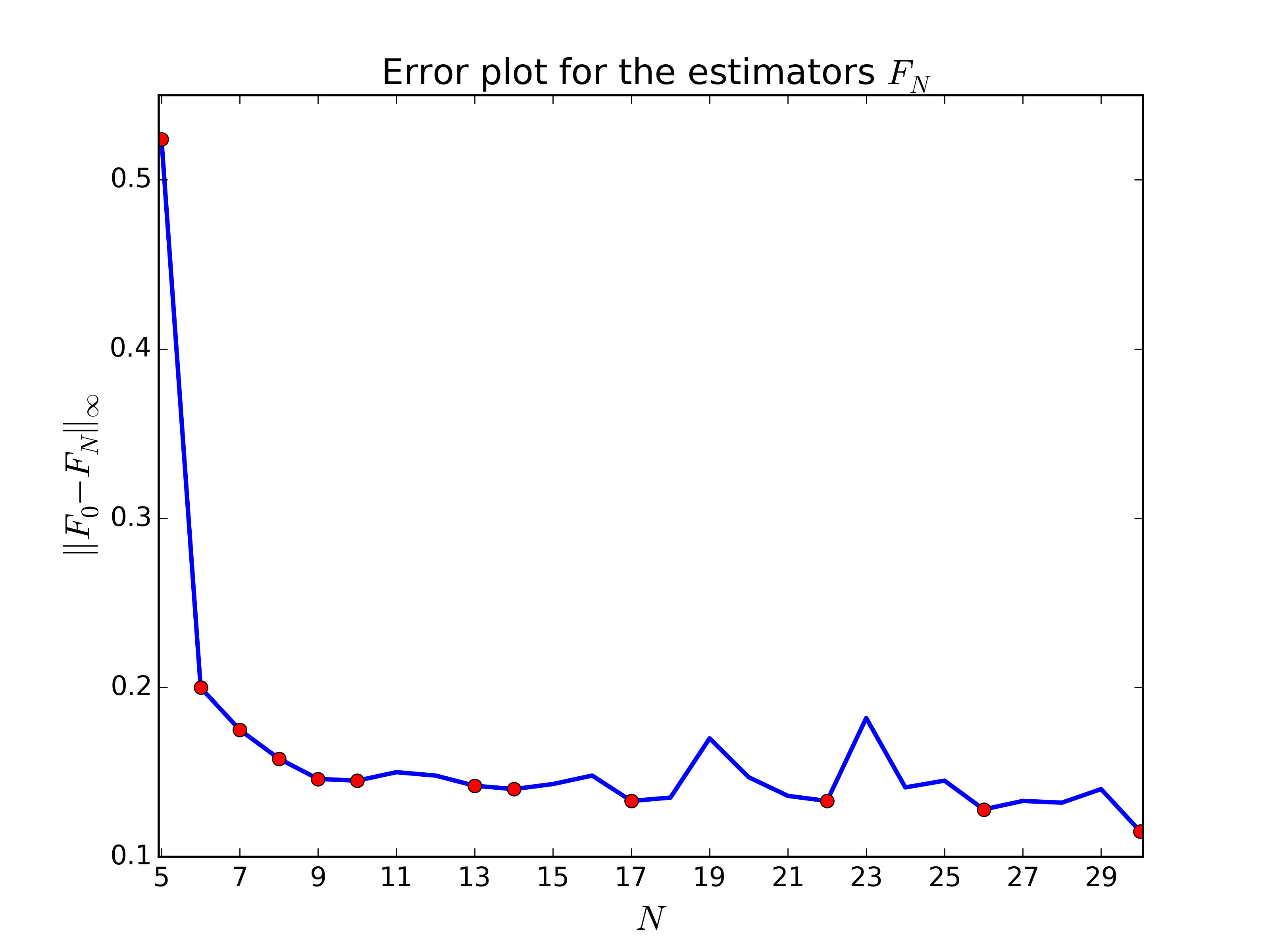}

}
\par\end{centering}

\centering{}\protect\caption{\label{fig:evidence-1}Simulation results for pseudo-data generated
with Beta distribution with $\alpha=\beta=0.5$ (a) Post-fragmentation
density used to generate test data. (b) Results of the optimization
scheme based on these data and an initial density function uniform
in $x$. (c) Absolute error for true probability measure $F_{0}$
and the approximate estimator $F_{30}$. (d) Error plots for the sequence
of estimators $\left\{ F_{N}\right\} _{N=5}^{30}$ . Solid red dots
indicate the convergent subsequence of the estimators. }
\end{figure}
To describe the aggregation within a laminar shear field (orthokinetic
aggregation) we used the kernel, 
\[
k_{a}(x,\,y)=10^{-6}\left(x^{1/3}+y^{1/3}\right)^{3}\,,
\]
proposed by von Smoluchowski \cite{Smoluchowski1917}. As in \cite{Bortz2008,AcklehFitzpatrick1997,Mirzaev2015a,Mirzaev2015b}
we assume that the breakage and removal rate of a floc of volume $x$
is proportional to its size,
\[
k_{f}(x)=10^{-1}x^{1/3}\quad\mu(x)=10^{-1}x^{1/3}\,.
\]
We also note that constants for the rate functions were chosen to
emphasize the fragmentation as a driving factor. The simulations were
run with initial size-distribution $b_{0}(x)=10^{3}\exp(-x)$ on $Q=[0,\,1]$
for $t\in[0,\,1]$. These data serve as the ``observed'' data $\mathbf{n}^{d}$.
Nonlinear constrained optimization employing the sequential least
squares algorithm as implemented in Python \texttt{fmin\_slsqp} was
used to minimize the cost functional in (\ref{eq:ApproxInv}). The
optimization was seeded with an initial density comprised of a uniform
density in $x$ for fixed $y$. Naturally, we constrained $\Gamma(\cdot,y)$
to be a probability density for each fixed $y$, i.e.,
\[
\int_{0}^{y}\Gamma(x;\,y)\,dx=1\text{ for all }y\in(0,\,\overline{x}].
\]
Our discretization uses $M=L=N_{x}=N$. We found that having more
comparison points in time was critical to observe the expected error
convergence in N. To illustrate this effect, we let $N_{t}=N+40$.
The result of the optimization for $N=30$, $F_{30}(x,y)$, is shown
in Figure~\ref{fig:evidence}. Absolute error for the \emph{true}
probability measure $F_{0}$ and the approximate estimator $F_{30}$
is depicted in Figure~ \ref{fig:abs_error_beta}. Convergence in
Prohorov metric implies the uniform convergence of probability measures
\cite{Gibbs2009}, i.e., 
\[
\sup_{(x,\,y)\in Q\times Q}\left|F(x,\,y)-\tilde{F}(x,\,y)\right|\le\rho(F,\,\tilde{F})\,.
\]
Towards this end, in Figure \ref{fig:error_beta_N}, we have illustrated
error plots for the sequence of estimators $\left\{ F_{N}\right\} _{N=5}^{30}$
. As it has been predicted in Theorem \ref{thm:Exist/conv of orig =000026 approx mins}
the sequence of the estimators has a subsequence (indicated by solid
red dots) which is convergent to \emph{true} probability measure $F_{0}$.
Our results for larger $N$ confirm that trend continues for $N>30$. 
\par\end{flushleft}

\section{Concluding Remarks}

Our efforts here are motivated by a class of mathematical models which
characterize a random process, such as fragmentation, by a probability
distribution. We are concerned with the inverse problem for inferring
the probability distribution, and present the specific problem for
the flocculation dynamics of aggregates in suspension which motivated
this study. We then developed the mathematical framework in which
we establish well-posedness of the inverse problem for inferring the
probability distribution. We also include results for overall method
stability for numerical approximation, confirming a computationally
feasible methodology. Finally, we verify that our motivating example
in flocculation dynamics conforms to the developed framework, and
illustrate its utility by identifying a sample distribution.

We originally proposed the flocculation model in \cite{Bortz2008}
and this work is one piece of a larger effort aimed at pushing the
boundaries for identifying microscale phenomena from size-structured
population measurements. In particular, we are interested in fragmentation.
The model proposed in \cite{ByrneEtal2011} uses knowledge of the
hydrodynamics to predict a breakage event and thus the post fragmentation
density $\Gamma$. With this work, we now have a tool to bridge the
gap between the experimental and modeling efforts for fragmentation.
And, a future paper will focus on using experimental evidence to validate
(or refute) our proposed fragmentation model.

\ack{}{The authors would like to thank J.~G.~Younger (University of Michigan)
for insightful discussions. This research was supported in part by
grants NIH-NIBIB-1R01GM081702-01A2, NSF-DMS-1225878, and DOD-AFOSR-FA9550-09-1-0404.}

\section*{References}{}

\bibliographystyle{siam}
\bibliography{bib_file}

\begin{thebibliography}{10}

\bibitem{Ackleh1997}
{\sc A.~S. Ackleh}, {\em {Parameter estimation in a structured algal
  coagulation-fragmentation model}}, Nonlinear Analysis, 28 (1997),
  pp.~837--854.

\bibitem{AcklehFitzpatrick1997}
{\sc A.~S. Ackleh and B.~G. Fitzpatrick}, {\em {Modeling aggregation and growth
  processes in an algal population model: analysis and computations}}, Journal
  of Mathematical Biology, 35 (1997), pp.~480--502.

\bibitem{Adams2003}
{\sc R.~Adams and J.~Fournier}, {\em {Sobolev spaces}}, Elsevier Ltd, Oxford,
  UK, 2003.

\bibitem{BablerEtal2008jfm}
{\sc M.~U. B\"{a}bler, M.~Morbidelli, and J.~Baldyga}, {\em {Modelling the
  breakup of solid aggregates in turbulent flows}}, Journal of Fluid Mechanics,
  612 (2008), pp.~261--289.

\bibitem{Lamb2009}
{\sc J.~Banasiak and W.~Lamb}, {\em {Coagulation, fragmentation and growth
  processes in a size structured population}}, Discrete and Continuous
  Dynamical Systems - Series B, 11 (2009), pp.~563--585.

\bibitem{BanksBihari2001}
{\sc H.~T. Banks and K.~L. Bihari}, {\em {Modeling and Estimating Uncertainty
  in Parameter Estimation}}, Inverse Problems, 17 (2001), pp.~95--111.

\bibitem{BanksBortz2005JIIP}
{\sc H.~T. Banks and D.~M. Bortz}, {\em {Inverse problems for a class of
  measure dependent dynamical systems}}, Journal of Inverse and Ill-posed
  Problems, 13 (2005), pp.~103--121.

\bibitem{htb1990}
{\sc H.~T. Banks and B.~G. Fitzpatrick}, {\em {Statistical methods for model
  comparison in parameter estimation problems for distributed systems}},
  Journal of Mathematical Biology, 28 (1990), pp.~501--527.

\bibitem{BanksKunisch1989}
{\sc H.~T. Banks and K.~Kunisch}, {\em {Estimation Techniques for Distributed
  Parameter Systems}}, vol.~1 of Systems \& Control: Foundations \&
  Applications, Birkh\"{a}user, Boston, MA, 1989.

\bibitem{BanksThompson2012}
{\sc H.~T. Banks and W.~C. Thompson}, {\em {Least Squares Estimation of
  Probability Measures in the Prohorov Metric Framework}}, N.C. State Center
  for Research in Scientific Computation Technical Report,  (2012).

\bibitem{Billingsley1968}
{\sc P.~Billingsley}, {\em {Convergence of Probability Measures}}, John Wiley
  \& Sons, New York, NY, 1968.

\bibitem{Bortz2008}
{\sc D.~M. Bortz, T.~L. Jackson, K.~A. Taylor, A.~P. Thompson, and J.~G.
  Younger}, {\em {Klebsiella pneumoniae Flocculation Dynamics}}, Bull. Math.
  Biology, 70 (2008), pp.~745--68.

\bibitem{ByrneEtal2011}
{\sc E.~Byrne, S.~Dzul, M.~Solomon, J.~Younger, and D.~M. Bortz}, {\em
  {Postfragmentation density function for bacterial aggregates in laminar
  flow}}, Physical Review E, 83 (2011), p.~041911.

\bibitem{Darzynkiewicz}
{\sc Z.~Darzynkiewicz, J.~P. Robinson, and H.~A. Crissman}, {\em {Flow
  Cytometry}}, in Methods in Cell Biology, Academic Press, San Diego, CA,
  2nd~ed., 1994, pp.~1--697.

\bibitem{DeVitaTheodore2008}
{\sc V.~T. DeVita, S.~{Lawrence, Theodore}, and S.~A. Rosenberg}, {\em {Cancer:
  Principles and Practice of Oncology}}, vol.~1, Lippincott Williams \&
  Wilkins, 2008.

\bibitem{Gamma1993}
{\sc C.~D. Gamma and C.~L. Jimeno}, {\em {Rock fragmentation control for
  blasting cost minimization and environmental impact abatement}}, in Rock
  Fragmentation by Blasting, P.~P. Roy, ed., A. A. Balkema Publishers,
  Amsterdam, The Netherlands, 1993, p.~273.

\bibitem{Gibbs2009}
{\sc A.~L. Gibbs and F.~E. Su}, {\em {On Choosing and Bounding Probability
  Metrics}}, International Statistical Review, 70 (2002), pp.~419--435.

\bibitem{HanEtal2003AICHEJ}
{\sc B.~Han, S.~Akeprathumchai, S.~R. Wickramasinghe, and X.~Qian}, {\em
  {Flocculation of biological cells: Experiment vs. theory}}, AIChE Journal, 49
  (2003), pp.~1687--1701.

\bibitem{IlanaElkinbVlodavsk2006}
{\sc N.~Ilana, M.~Elkinb, and I.~Vlodavsky}, {\em {Regulation, function and
  clinical significance of heparanase in cancer metastasis and angiogenesis}},
  The International Journal of Biochemistry \& Cell Biology, 38 (2006),
  p.~2018.

\bibitem{Kelley1955}
{\sc J.~L. Kelley}, {\em {General Topology}}, Van Nostrand-Reinhold, Princeton,
  NJ, 1955.

\bibitem{Mirzaev2015a}
{\sc I.~Mirzaev and D.~M. Bortz}, {\em {Criteria for linearized stability for a
  size-structured population model}}, arXiv:1502.02754,  (2015).

\bibitem{Mirzaev2015b}
\leavevmode\vrule height 2pt depth -1.6pt width 23pt, {\em {Stability of steady
  states for a class of flocculation equations with growth and removal}},
  arXiv:1507.07127,  (2015).

\bibitem{Persson1994}
{\sc P.-A. Persson, R.~Holmberg, and J.~Lee}, {\em {Rock Blasting and
  Explosives engineering}}, CRC Press, 1994.

\bibitem{Smoluchowski1917}
{\sc M.~van Smoluchowski}, {\em {Versuch einer mathematischen theorie der
  koagulation kinetic kolloider losungen}}, Zeitschrift f\"{u}r physikalische
  Chemie, 92 (1917), pp.~129--168.

\bibitem{Wyckoff2000}
{\sc J.~B. Wyckoff, J.~G. Jones, J.~S. Condeelis, and J.~E. Segall}, {\em {A
  critical step in metastasis: in vivo analysis of intravasation at the primary
  tumor.}}, Cancer research, 60 (2000), pp.~2504--2511.

\end{thebibliography}

\end{document}